\documentclass[12pt,fullpage,leqno]{article}
\usepackage[english]{babel}
\usepackage{amsmath}
\numberwithin{equation}{section}
\usepackage[margin=1in]{geometry}

\usepackage{amsfonts}
\usepackage{amsmath}
\usepackage{epsfig}
\usepackage{graphicx}
\usepackage{color}
\usepackage{lmodern} 
\usepackage{mathrsfs}
\usepackage{amsfonts}
\usepackage{amssymb}
\usepackage{mathdots}
\usepackage{amsmath,amssymb,cases,amsthm,color}
\usepackage{float}
\usepackage{subfig}
\usepackage[overload]{empheq}
\usepackage{mathtools}

\usepackage{perpage} 
\MakePerPage{footnote}

\usepackage{fancyhdr} 

\usepackage{indentfirst}

\usepackage{enumerate}
\usepackage{enumitem}
\usepackage[title,titletoc,toc]{appendix}

\newtheorem{theo}{\bf Theorem}[section]
\newtheorem{lem}[theo]{\bf Lemma}
\newtheorem{pro}[theo]{\bf Proposition}

\newtheorem{defi}[theo]{\bf Definition}

\theoremstyle{definition}
\newtheorem{ex}{Example}[section]
\newtheorem{rem}{Remark}[section]

\newcommand{\LH}{L_1}
\newcommand{\Lb}{L_{\beta}}

\newenvironment{Proofc}[1]{\smallskip\par\noindent\textsc{#1}\quad}%
  {\hfill$\Box$\bigskip\par}

\nonstopmode

\makeatletter
\newcommand{\gl@align}[2]{\lower.6ex\vbox{\baselineskip\z@skip\lineskip\z@
\ialign{$\m@th#1\hfil##\hfil$\crcr#2\crcr=\crcr}}}
\newcommand{\Leqq}{\mathrel{\mathpalette\gl@align<}}
\newcommand{\Geqq}{\mathrel{\mathpalette\gl@align>}}
\makeatother

\begin{document}
\title{\bf A rigorous setting for the reinitialization of first order level set equations} 
\author{Nao Hamamuki\footnote{Department of Mathematics, Hokkaido University, Kita 10, Nishi 8, 
Kita-Ku, Sapporo, Hokkaido, 060-0810, Japan. e-mail: hnao@math.sci.hokudai.ac.jp} , Eleftherios Ntovoris\footnote{CERMICS - ENPC 
6 et 8 avenue Blaise Pascal 
Cit\'e Descartes - Champs sur Marne 
77455 Marne la Vall\'ee Cedex 2, France.
e-mail: ntovorie@cermics.enpc.fr}} 
\maketitle

\begin{abstract}
In this paper we set up a rigorous justification for the reinitialization algorithm. 
Using the theory of viscosity solutions,
we propose a well-posed Hamilton-Jacobi equation with a parameter,
which is derived from homogenization for a Hamiltonian 
discontinuous in time which appears in the reinitialization.
We prove that, as the parameter tends to infinity,
the solution of the initial value problem converges to 
a signed distance function to the evolving interfaces.
A locally uniform convergence is shown when the distance function is continuous,
whereas a weaker notion of convergence is introduced 
to establish a convergence result to a possibly discontinuous distance function. 
In terms of the geometry of the interfaces,
we give a necessary and sufficient condition for the continuity of the distance function.
We also propose another simpler equation 
whose solution has a gradient bound away from zero.
\end{abstract}



\tableofcontents

\section{Introduction}

\paragraph{Setting of the problem}

In this paper we establish a rigorous setting for the reinitialization algorithm. 
In the literature ``reinitialization" usually refers to 
the idea of stopping the process of solving an evolution equation regularly in time
and changing its solution at the stopping time
so that we obtain a function which approximates 
the (signed) distance function to the zero level set of the solution.
A typical example of such evolution equations is 
\begin{equation}\label{intro:eqevol}
u_t=c(x,t)|\nabla u|,
\end{equation}
where $u=u(x,t)$ is the unknown, $u_t=\partial_t u$, 
$\nabla u=(\partial_{x_1}u, \dots ,\partial_{x_n}u)$ and
$| \cdot |$ stands for the standard Euclidean norm in $\mathbf{R}^n$.
The equation \eqref{intro:eqevol} describes
a motion of an interface $\Gamma_t$ in $\mathbf{R}^n$
whose normal velocity is equal to $c=c(x,t)$,
where at each time the zero level set of $u(\cdot,t)$ represents the interface $\Gamma_t$.
In general, the solution of \eqref{intro:eqevol} does not preserve the distance function, 
and its gradient can get very close to zero.
For example, the function
\[ u(x,t)=1-|x|\mathit{e}^{-t} \]
solves the problem
\[ \left\{
\begin{aligned}
& u_t  = |x|\cdot |u_x| &\text{ in }& \mathbf{R}\times (0,T), \\
& u(x,0)=1-|x|& \text{ in } &\mathbf{R}
\end{aligned}
\right. \]
in the viscosity sense.
For the numerical study of \eqref{intro:eqevol}
several simplifications can be made
when the solution is or approximates the distance function.
One of the reasons is the fact that the gradient of the distance function 
is always 1 and thus bounded away from 0. 
When the gradient degenerates like in the above example,
it becomes difficult to compute precisely the zero level sets.
The reinitialization is used to overcome such an issue.
For a more detailed discussion on the numerical profits of the reinitialization, 
see \cite{Sethian.1999, Osher-Fedkiw.2003}.

Several reinitialization techniques have been introduced in the literature. 
In this paper we focus on the one introduced by 
Sussman, Smereka and Osher (\cite{Sussman-Smereka-Osher.1994}). Their method allows to reinitialize \eqref{intro:eqevol} 
without explicitly computing the signed distance function 
with the advantage that the level set function of their method 
approximates the signed distance at every time.

We briefly explain the main idea of the method in \cite{Sussman-Smereka-Osher.1994}.
Consider the {\em corrector equation}
\begin{equation}\label{intro:eqcor}
\phi _t=\mathrm{sign} (\phi)(1-|\nabla \phi|),
\end{equation}
where $\mathrm{sign}(\cdot)$ is the sign function defined as
\[ \mathrm{sign}(r)=\left\{
\begin{aligned}
& \frac{r}{|r|}& &\text{if } r\neq 0,& \\
& 0& &\text{else.}& 
\end{aligned}
\right. \]
The solution of this equation asymptotically converges to 
a steady state $|\nabla\phi |=1$, 
which is a characteristic property of the distance function;
see Subsection \ref{section_compar_eikonal}. 
The purpose of the sign function in \eqref{intro:eqcor} is 
to control the gradient.
In the region where $\phi$ is positive,
the equation is $\phi _t=1-|\nabla \phi |$.
Thus, the monotonicity of $\phi$ is prescribed by the order of $1$ and $|\nabla \phi |$.
This forces $|\nabla \phi|$ to be close to 1 as time passes.
Also, the relation $\mathrm{sign}(0)=0$ guarantees that
the initial zero level set is not distorted
since $\phi_t=0$ on the zero level.
Roughly speaking, the idea of \cite{Sussman-Smereka-Osher.1994} is 
to stop the evolution of \eqref{intro:eqevol} periodically in time and
solve \eqref{intro:eqcor} till convergence to the signed distance function is achieved. 
This method was first applied in \cite{Sussman-Smereka-Osher.1994} 
for the calculation of the interface of a fluid flow, 
with the disadvantage that the fluid flow can lose mass, 
because of the accumulation of numerical errors after many periods are completed. 
This problem was later fixed in \cite{Sussman-Fatemi.1999}.

Up to the authors' knowledge there is no rigorous setting 
for the reinitialization process described above. 
In this paper we study an evolution of an interface $\Gamma_t$
given as the zero level set of the solution $u$
to the initial value problem of the general Hamilton-Jacobi equation
\begin{equation}\label{intro:eqgen}
u_t=H_1(x,t,\nabla u)
\end{equation}
with a Lipschitz continuous initial datum $u_0$. 
Here $H_1=H_1(x,t,p)$ is assumed to be continuous, geometric 
and Lipschitz continuous in $x$ and $p$.
These assumptions are often used in the literature to guarantee 
that \eqref{intro:eqgen} is well-posed. 
As a corrector equation we use a slight modification of \eqref{intro:eqcor}, namely
\begin{equation}\label{intro:eqcor2}
u_t=\frac{u}{\sqrt{\varepsilon _0 ^2+u^2}}h(\nabla u),
\end{equation}
where $\varepsilon _0>0$ is fixed 
and the function $h$ can be one of the following:
\begin{enumerate}[label=(\arabic*) ,ref=\arabic*]
	\item\label{intro:itm1} $h(p)=1-|p|$,
	\item\label{intro:itm2} The plus part of \eqref{intro:itm1}, i.e., $h(p)=(1-|p|)_+$.
\end{enumerate}
The function $\beta (u)=u/\sqrt{\varepsilon _0 ^2 +u^2}$ 
is a smoother version of the sign function. 
Although the function $h$ in \eqref{intro:itm2} 
does not preserve the distance function 
in the sense of \cite{Sussman-Smereka-Osher.1994} and 
in a way that will be made rigorous later in Theorem \ref{Theo1.1} and Example \ref{exmp:no_convergence}, 
it does however prevent the gradient of the solution to approach zero on the zero level set. 
Moreover, it provides a simple monotone scheme 
for the numerical solution of the problems which we will encounter. 
In fact, our result applies for corrector equations 
which are more general than \eqref{intro:eqcor2}, 
but for the sake of simplicity we present, in this section, 
the main idea for this model equation.

The idea, as in \cite{Sussman-Smereka-Osher.1994}, 
is to solve \eqref{intro:eqgen} and \eqref{intro:eqcor2} periodically in time, 
the first for a period of $k_1\Delta t$ and the second for $k_2\Delta t$, 
where $k_1,k_2,\Delta t>0$ and one period will be completed 
at a time step of length $\varepsilon =(k_1+k_2)\Delta t$. 
We are thus led to define the following combined Hamiltonian
\[ H_{12}(x,t,\tau,r,p):=\begin{cases}
H_1(x,\frac{t}{1+\frac{k_2}{k_1}},p)
& \mbox{if} \ (i-1) < \tau \leq (i-1) + \frac{k_1 \Delta t}{\varepsilon}, \\
\frac{u}{\sqrt{\varepsilon _0 ^2+u^2}}h(\nabla u)
& \mbox{if} \ (i-1)+ \frac{k_1 \Delta t}{\varepsilon} < \tau \leq i
\end{cases} \]
for $i=1,..., \lceil \frac{T}{\varepsilon}\rceil$.
Here by $\lceil x \rceil$
we denote the smallest integer which is not smaller than $x\in \mathbf{R}$.
The rescaling of the Hamiltonian $H_1$ in time is required 
since certain time intervals are reserved for the corrector equation.
More precisely, $H_1$ is solved in time length 
$k_1\Delta t\lceil \frac{T}{\varepsilon}\rceil\sim T\frac{k_1}{k_1+k_2}=\frac{T}{1+\frac{k_2}{k_1}}$.
One would expect that solving the two equations infinitely often 
would force the solution of the reinitialization algorithm 
to converge to the signed distance function to $\Gamma_t$;
we denote it by $d$. 
Therefore we are led to study the limit as $\varepsilon \rightarrow 0$ of the solutions of 
\begin{equation}\label{intro:eqHJe}
\left\{
\begin{aligned}
& u^\varepsilon _t
=H_{12} \left( x,t,\frac{t}{\varepsilon},u^\varepsilon ,\nabla u^\varepsilon \right) &\text{ in }& \mathbf{R}^n\times (0,T), \\
& u^\varepsilon (x,0)=u_0(x) & \text{ in } &\mathbf{R}^n.
\end{aligned}
\right.
\end{equation}
This is a homogenization problem with the Hamiltonian $H_{12}$ 
being 1-periodic and discontinuous in the fast variable $\tau =t/\varepsilon$. 
Since the limit above is taken for $\Delta t\rightarrow 0$ 
(and consequently $\varepsilon \rightarrow 0$), 
two free parameters still remain, namely $k_1$ and $k_2$. 
In fact, we show that the solutions of \eqref{intro:eqHJe} converge, 
as $\varepsilon \rightarrow 0$ and after rescaling, to the solution $u^{\theta}$ of 
\begin{equation}\label{intro:eqtheta}
\left\{
\begin{aligned}
& u_t ^\theta = H_1(x,t,\nabla u^\theta )+\theta \beta (u^{\theta})h(\nabla u^\theta ) 
&\text{ in }& \mathbf{R}^n\times (0,T),  \\
& u^\theta (x,0)=u_0(x) &\text{ in }& \mathbf{R}^n.
\end{aligned}
\right.
\end{equation}
Here $\theta =k_2/k_1$ is the ratio of length of the time intervals 
in which the equations \eqref{intro:eqgen} and \eqref{intro:eqcor2} are solved. 
If we solve the corrector equation \eqref{intro:eqcor2} 
in a larger interval than the one we solve the original \eqref{intro:eqgen}, 
we can expect the convergence to a steady state. 
For this reason we study the limit as $\theta \rightarrow \infty$ 
of the solutions of \eqref{intro:eqtheta}.

Let us consider the function $h$ in \eqref{intro:itm1} as the model case.
Roughly speaking, the limit $\theta \to \infty$ forces $h(\nabla u^{\theta})$
to be close to 0 except on the zero level of $u^{\theta}$,
i.e., $|\nabla u^{\theta}| \approx 1$ for large value of $\theta$.
If we further know that the zero level set of $u^{\theta}$ is
the same as that of the solution of \eqref{intro:eqgen}
and hence is equal to $\Gamma_t$
(we call this property a preservation of the zero level set),
then we would get a convergence of $u^{\theta}$ to the signed distance function $d$,
which is known to be a solution of the eikonal equation
\begin{equation} 
|\nabla d|=1
\label{intro:eik}
\end{equation}
with the homogeneous Dirichlet boundary condition on the zero level.
The preservation of the zero level set for \eqref{intro:eqtheta} 
mainly follows from \cite{Nao-thesis}.

To justify the convergence to $d$ rigorously,
the comparison principle for the eikonal equation \eqref{intro:eik}
is used to compare the distance function and 
a half-relaxed limit of $u^{\theta}$,
which is a weak notion of the limit for a sequence of functions.
To do this, we need to know that the limit of $u^{\theta}$
also preserves the zero level set.
This is not clear, despite the fact that $u^{\theta}$ always preserves the zero level set for every $\theta>0$.
For the preservation of the zero level set by the limit, 
continuity of the distance function plays an important role.
As is known, if we fix a time, 
$d(\cdot,t)$ is a Lipschitz continuous function,
but $d$ is not continuous in general as a function of $(x,t)$.
Indeed, when the interface has an extinction point (Definition \ref{defn:Expt}),
the distance function can be discontinuous near this point.
For our problem, by constructing suitable barrier functions
it turns out that, when $d$ is continuous,
the zero level set of the half-relaxed limit of $u^{\theta}$ 
is the same as $\Gamma_t$.
Consequently, we obtain the locally uniform convergence of $u^{\theta}$ to $d$;
see Theorem \ref{Theor1} \eqref{itm:main2}.

Concerning the locally uniform convergence,
we further consider a condition which guarantees the continuity of $d$.
An important property of first order equations 
is the finite speed of propagation 
(Subsection \ref{Section:finite_propagat}), which allows us to show 
that the only way the distance function can be discontinuous is 
if points at the zero level extinct instantaneously. 
More precisely, we show  
that the distance function is continuous at $(x,t)$ if and only if 
at least one of the nearest points of $x$ to $\Gamma_t$
is a non extinction point;
see Theorem \ref{thm:cont_d} \eqref{itm:continuity_3}.
Therefore, if the latter condition is satisfied 
for every $(x,t) \in \mathbf{R}^n \times (0,T)$,
then the solutions $u^{\theta}$ of \eqref{intro:eqtheta} converge 
locally uniformly to $d$ in $\mathbf{R}^n\times (0,T)$
(Remark \ref{rem:conection_main_theorem}).
The converse is also true.

If the signed distance function $d$ is discontinuous, 
we cannot expect that the continuous solutions $u^{\theta}$ of \eqref{intro:eqtheta}
will converge locally uniformly to $d$.
In fact, when $d$ is discontinuous,
the zero level sets of the half-relaxed limit of $u^{\theta}$ 
are not $\Gamma_t$,
and this prevents us to apply the comparison principle for \eqref{intro:eik}.
We can however show (Theorem \ref{Theor1} \eqref{itm:main1}) 
a weaker notion of convergence to $d$; 
namely a convergence to $d$ from below in time as follows:
\begin{equation} 
\lim_{\begin{subarray}{c}(y,s,\theta) \to (x,t,\infty) \\ s\leq t \end{subarray}} u^{\theta}(y,s)=d(x,t) 
\quad \mbox{for all $(x,t) \in \mathbf{R}^n \times (0,T)$}. 
\label{eq:result_below}
\end{equation}
This will be shown by introducing a notion of a half-relaxed limit from below in time
and by using the fact that $d$ is continuous from below in time.
The result \eqref{eq:result_below} also implies 
a locally uniform convergence at any fixed time, that is,
$u^\theta (\cdot ,t)$ converges to $d(\cdot ,t)$ locally uniformly in $\mathbf{R}^n$ as $\theta \rightarrow \infty$;
see Theorem \ref{Theor1} \eqref{itm:main1.5}. 

In a future work we plan to introduce a numerical scheme for \eqref{intro:eqtheta}, 
where no reinitialization will be required. 
We also plan to study numerically and rigorously 
a similar method for second order equations, 
including the mean curvature flow, of the form
\[ u_t=H_1(x,t,\nabla u,\nabla ^2 u)+\theta \beta (u)h(\nabla u). \]

\paragraph{Review of the literature}
In \cite{Chopp.1993}, Chopp used a reinitialization algorithm for the mean curvature flow. 
His technique does not utilize a corrector equation, instead, 
at each stopping time he recalculates the signed distance and 
starts the evolution again with this new initial value. 
In \cite{Sethian.1999} it is mentioned that another way to reinitialize
is to compute the signed distance at each stopping time using the fast marching method.
In \cite{Sussman-Fatemi.1999} the problem of the movement of 
the zero level set which appeared in \cite{Sussman-Smereka-Osher.1994} 
is fixed by solving an extra variational problem during the iteration of the two equations. 
Sethian in \cite{Sethian.1999} suggests to use, 
instead of the reinitialization algorithm, 
the method of extended velocity described in chapter 11.

In \cite{Delfour-Zolesio.2004} a new nonlinear equation is introduced
for the evolution of open sets with thin boundary under a given velocity field. 
The solution is for every time the signed distance function to the boundary of the open set 
(called an oriented distance function in \cite{Delfour-Zolesio.2004}).
Other numerical methods for preserving the signed distance are presented in \cite{Estellers-Zosso-Lai-Osher-Thiran-Bresson.2012} and \cite{Gomes-Faugeras.1999}.

In \cite{Goto_Ishii_Ogawa}, the authors use the approximation of the mean curvature flow by the Allen-Cahn equation and they prove the convergence of an equation to the signed distance function. See also \cite{Chambolle_Novaga} for a related theory developed for anisotropic and crystalline mean curvature flow.

\paragraph{Summary}
To sum up, the contributions of this paper are:

\begin{itemize}
\item 
mathematical justification of the reinitialization procedure,

\item 
introduction of a new approximate scheme 
for the distance function of evolving interfaces, i.e.,
solving \eqref{intro:eqtheta} and taking the limit as $\theta \rightarrow \infty$,

\item 
formulation of a necessary and sufficient condition 
for the solution of the scheme \eqref{intro:eqtheta} to converge locally uniformly
to the signed distance function, 
in terms of topological changes of interfaces,

\item
discovery of a weak notion of a limit 
which gives the signed distance function even if it is discontinuous.

\end{itemize}

We also mention that

\begin{itemize}

\item 
through the rigorous analysis of the reinitialization procedure,
we retrieve the correct rescaling in time 
of the equation \eqref{intro:eqgen} in order to approximate
the signed distance function,
and thus we extend the reinitialization procedure 
to evolutions with time depending velocity fields,

\item 
lastly, the equation in \eqref{intro:eqtheta} 
with $h$ satisfying \eqref{intro:itm2} or 
more generally (as we will see later) the assumption \eqref{eq:hprop3}
admits a natural numerical scheme with a CFL condition, 
see also \cite{MR1115933} or \cite{Souganidis}. 
We plan to study this last part in a future work.
\end{itemize}

\paragraph{Organization of the paper}
In Subsection \ref{Sub:main_theorems} we state the main results, 
and in Subsection \ref{section:prepar} we present known results 
concerning a well-posedness and regularity of viscosity solutions.
Section \ref{section:main_tools} consists of main tools 
which we use in order to prove our main theorems,
namely the preservation of the zero level set (Subsection \ref{subsect:preservation}), 
construction of barrier functions (Subsection \ref{Sect:Barrier}) 
and characterization of the distance function via the eikonal equation (Subsection \ref{section_compar_eikonal}).
In Section \ref{section:converg_results} we prove convergence results 
to the signed distance function $d$.
The proof for continuous $d$ and that for discontinuous $d$ 
will be given separately.
Section \ref{Section:continuity_distance} is concerned with
continuity properties of the distance function,
and finally in Section \ref{section:homogen} 
we prove a homogenization result.

\section{Main results}\label{section:main_results}

\subsection{Main theorems}\label{Sub:main_theorems}
We study the evolution of the zero level set of a function $w$ given by the following problem:
\begin{subequations}\label{eq:init}
\begin{alignat}{2}[left=\empheqlbrace]
 & w_t = H_1(x,t,\nabla w) &\quad & \text{in }\mathbf{R}^n \times (0,T),\label{eq:initeq} \\[\medskipamount]
 & w(x,0)=u_0(x) & & \text{in } \mathbf{R}^n.\label{eq:initin}
\end{alignat}
\end{subequations}
Here $u_0$ is a possibly unbounded Lipschitz continuous function on $\mathbf{R}^n$ and its Lipschitz constant is denoted by $L_0$. The function $H_1=H_1(x,t,p):\mathbf{R}^n \times [0,T] \times \mathbf{R}^n \rightarrow \mathbf{R}$ satisfies
 \begin{enumerate}[label=(H\arabic*) ,ref=H\arabic*]
\setcounter{enumi}{0}
    \item\label{itm4} $H_1\in C(\mathbf{R}^n\times [0,T]\times \mathbf{R}^n)$,
	\item\label{itm5} $H_1(x,t,\lambda p )=\lambda H_1(x,t,p)$ for all $\lambda >0$, $x\in \mathbf{R}^n$, $t\in [0,T]$ and $p\in \mathbf{R}^n$,
	\item\label{itm6} There is a positive constant $L_1$ such that 
	$$|H_1(x,t,p)-H_1(y,t,p)|\leq L_1|x-y|$$
	for all $x,y\in \mathbf{R}^n,\, t\in [0,T]$ and $p\in \mathbf{R}^n$ with $|p|=1$,
	\item \label{itm6.1} There is a positive constant $L_2$ such that 
	\[
	|H_1(x,t,p)-H_1(x,t,q)|\leq L_2|p-q|
	\] 
	for all $x\in \mathbf{R}^n,\, t\in [0,T]$ and $p,q\in \mathbf{R}^n$,
	\item\label{itm7} $\displaystyle\sup_{\begin{subarray}{c}(x,t,p)\in \mathbf{R}^n \times [0,T]\times \mathbf{R}^n \\ |p|=1 \end{subarray}}|H_1(x,t,p)|<\infty$.
	 \end{enumerate} 
 \begin{rem}\label{main:lipest}
 In order to get a more precise estimate for the Lipschitz constant of solutions considered in this paper, we will use in Proposition \ref{prop2} instead of \eqref{itm6} the following:
 \begin{enumerate}[label=(H\arabic*-s) ,ref=H\arabic*-s]
 \setcounter{enumi}{2}
\item\label{itm6'} There is a function $D\in C([0,T])$ such that
 \[
  |H(x,t,p)-H(y,t,p)|\leq D(t)|x-y|
 \]
 for all $x,y\in \mathbf{R}^n,\, t\in [0,T]$ and $p\in \mathbf{R}^n$ with $|p|=1$.
 \end{enumerate}
 Note that the assumption \eqref{itm6'} implies the assumption \eqref{itm6}.
 \end{rem}
 \begin{rem}
 In the literature the assumption \eqref{itm6} is usually given as: There are $L,\bar{L}$ positive, such that 
 \begin{equation}\label{Genlipcond}
|H(x,t,p)-H(y,t,p)|\leq L|x-y||p|+\bar{L}|x-y|
 \end{equation}
 for all $x,y,p\in \mathbf{R}^n$ and $t\in [0,T]$.
 However, since the Hamiltonian $H$ is geometric (the assumption \eqref{itm5}), it turns out that the conditions \eqref{itm6} and \eqref{Genlipcond} are equivalent.
  Indeed, it is clear that \eqref{itm6} implies \eqref{Genlipcond} with $L=L_1$ and $\bar{L}=0$. Also, under \eqref{Genlipcond} we can easily derive \eqref{itm6} with $L_1=L+\bar{L}$. Here let us also show that, in fact, we can take $L_1=L$ in \eqref{itm6} when \eqref{Genlipcond} holds. Let $x,y\in \mathbf{R}^n$, $t\in [0,T]$, $p\in \mathbf{R}^n$ with $|p|=1$ and $r>0$. Then we have
 \[
 |H(x,t,rp)-H(y,t,rp)|\leq Lr|x-y|+\bar{L}|x-y|.
 \]
 Dividing both sides by $r$ and using \eqref{itm5}, we get
 \[
 |H(x,t,p)-H(y,t,p)|\leq L|x-y|+\bar{L}\frac{|x-y|}{r}.
 \]
 If we now take the limit as $r\rightarrow +\infty $, we get \eqref{itm6} with $L_1=L$.
  \end{rem}
 The assumption \eqref{itm5} is natural for a geometric evolution problem, while \eqref{itm6} is used for construction of barriers in Subsection \ref{Sect:Barrier} and for the proof of Lipschitz continuity of solutions in Appendix \ref{appendix}. We call the constant $L_2$ in assumption \eqref{itm6.1} the speed of propagation of the zero level set of the solutions.
 Existence, uniqueness and other properties of the problem \eqref{eq:init} can be found in Subsection \ref{section:prepar}.
 Since the zero level set of the solution $w$ of \eqref{eq:init} is the main focus of this paper we will use the following notations for $t\in [0,T)$:
 \begin{align}
 D_t ^{\pm} &:=\{ x\in \mathbf{R}^n \mid\, \pm w(x,t)\geq 0 \},\label{zeronot1}\\
 \Gamma _t &:=\{ x\in \mathbf{R}^n \mid\, w(x,t)=0\}\label{zeronot2}
 \end{align}
 and
 \[ D^+:=\bigcup_{t\in (0,T)} (D^+_t \times \{ t \}), \quad 
D^-:=\bigcup_{t\in (0,T)} (D^-_t \times \{ t \}). \]
In what follows we will always suppose that the evolution associated with $w$ is not empty, i.e.,
\begin{equation}\label{eq:sec_2_1}
\Gamma _t\neq \emptyset \text{ for all } t\in [0,T).
\end{equation}
Also for $\Omega\subset \mathbf{R}^n$ the distance function $\mathrm{dist}(\cdot ,\Omega):\mathbf{R}^n\rightarrow [0,\infty )$ is defined as
\[
\mathrm{dist}(x,\Omega):=\inf_{y \in \Omega} |x-y|.
\]
\begin{rem}
It is well-known that the distance function 
$\mathrm{dist}(\cdot, \Omega)$ is Lipschitz continuous. 
Indeed, for $x,y \in \mathbf{R}^n$ and $z \in \Omega$, we have 
\[ |y-z| \le |x-z|+|x-y| \quad \mbox{and} \quad 
|x-z| \le |y-z|+|x-y|. \]
If we take the infimum for $z \in \Omega$, 
the above inequalities become
\[ |\mathrm{dist}(x, \Omega)-\mathrm{dist}(y, \Omega)| \le |x-y|. \]
\end{rem}
 For a function $w:\mathbf{R}^n \times [0,T) \to \mathbf{R}$,
$\mathrm{Lip}_x[w]$ stands for the Lipschitz constant of $w$ with respect to $x$, i.e.,
\[ \mathrm{Lip}_x[w]
:=\sup_{\begin{subarray}{c} x,y \in \mathbf{R}^n \\ x \neq y \end{subarray}}
\sup_{t \in [0,T)}\frac{|w(x,t)-w(y,t)|}{|x-y|} 
\in [0,\infty]. \]
      \quad Our first result concerns an equation of the form
 \begin{equation}\label{eq6}
u_t ^\theta = H_1(x,t,\nabla u^\theta )+\theta H_2(u^\theta ,\nabla u^\theta )\quad \text{ in } \mathbf{R}^n\times [0,T),
\end{equation}
where $\theta >0$ is a parameter, $H_1$ is as in \eqref{eq:initeq} and 
\begin{equation}\label{eq1.0*}
H_2(r,p)=\beta (r)h(p).
\end{equation}
The function $\beta$ is assumed to satisfy
\begin{enumerate}[label=(B) ,ref=B]
	\item\label{eq1.1*} $\mathrm{Lip}[\beta ]=:L_\beta <\infty$ and $\beta$ is non-decreasing and bounded in $\mathbf{R}$
with $\beta(0)=0$, $\beta(r)>0$ if $r>0$, $\beta(r)<0$ if $r<0$,
 \end{enumerate}
where by $\mathrm{Lip}[f]$ we denote the Lipschitz constant of a function $f:\mathbf{R}^n\rightarrow \mathbf{R}$. Moreover, $h:\mathbf{R}^n\rightarrow \mathbf{R}$ is such that
\begin{equation}\label{hprop1}
\text{there is a modulus } \omega _h \text{ such that } |h(p)-h(q)|\leq \omega _h (|p-q|)\text{ for all } p,q\in \mathbf{R}^n .
\end{equation}
Here a function $\omega :[0,\infty )\rightarrow [0,\infty )$ is called a modulus if $\omega$ is non-decreasing and $0=\omega (0)=\displaystyle\lim_{r\rightarrow 0}\omega (r)$.
We will also use one of the following assumptions for the function $h$:
\begin{subequations}\label{eq:hprop2}
\begin{alignat}{2}[left=\empheqlbrace]
 &h(p)>0 &\text{ if } & |p|<1\label{eq:hprop2.1}, \\[\medskipamount]
 & h(p)<0 &\text{ if } & |p|>1 \label{eq:hprop2.2}
\end{alignat}
\end{subequations}
or 
\begin{subequations}\label{eq:hprop3}
\begin{alignat}{2}[left=\empheqlbrace]
 &h(p)>0 &\text{ if } & |p|<1\label{eq:hprop3.1}, \\[\medskipamount]
 & h(p)= 0 &\text{ if } & |p|\geq 1 \label{eq:hprop3.2}.
\end{alignat}
\end{subequations}

Examples of these functions are
\begin{align}
& H_1(x,t,p)=c(x,t)|p|,\label{eq:1}\\
& H_2(u,p)=\frac{u^2}{\sqrt{\varepsilon _0 ^2+u^2}}h(p) \label{eq:2}
\end{align}
for 
\begin{equation}\label{eq:3}
h(p)=1-|p|
\end{equation}
or
\begin{equation}\label{eq:4}
h(p)=(1-|p|)_+,
\end{equation}
where $\varepsilon _0>0$, $c$ is Lipschitz continuous with respect to $x\in\mathbf{R}^n$ uniformly in time, and for $a\in \mathbf{R}$ we denote by 
\[
a_{\pm}=\max\{ \pm a,0\}
\]
the positive and negative part of $a$.
 We see that the function $h$ defined in \eqref{eq:3} satisfies \eqref{eq:hprop2} while \eqref{eq:4} satisfies \eqref{eq:hprop3}.\par
For a function $w(x,t)$ defined in $\mathbf{R}^n\times [0,T)$, we define the signed distance function $d(x,t)$, from the zero level set of $w$, as follows:
\begin{equation}\label{eq1*}
d(x,t)=
\begin{cases} 
\mathrm{dist}(x,\Gamma _t)   & \text{ if } x\in D_t ^+\cup \Gamma _t, \\
 -\mathrm{dist}(x,\Gamma _t) & \text{ if } x\in D_t ^-. 
\end{cases}
\end{equation}
Here $D^{\pm} _t$ and $\Gamma _t$ are defined in \eqref{zeronot1} and \eqref{zeronot2}.\par 
For later use we collect our main assumptions in the following list:
\begin{equation} 
\begin{cases} 
\mbox{$u_0$ is Lipschitz continuous in $\mathbf{R}^n$, \ 
$H_1$ satisfies (H1)--(H5),} \\
\mbox{$H_2$ is of the form (2.7), \ 
$\beta$ satisfies (B), \ 
$h$ satisfies (2.8).}
\end{cases} 
\label{eq:theor1_assumpt}
\end{equation}
For the solution $u^\theta$ of \eqref{eq6} and \eqref{eq:initin} we have the following main theorem.
\begin{theo}[Convergence of $u^\theta $ to the signed distance function]\label{Theor1}
Assume \eqref{eq:theor1_assumpt} and \eqref{eq:hprop2}. Let $u^\theta$ be the solution of \eqref{eq6} and \eqref{eq:initin}. Let $d$ be the signed distance function as in \eqref{eq1*}.
  Then
\begin{enumerate}[label=(\roman*) ,ref=\roman*]
\setcounter{enumi}{0}
    \item\label{itm:main1} \[ \lim_{\begin{subarray}{c}(y,s,\theta) \to (x,t,\infty) \\ s\leq t \end{subarray}} u^{\theta}(y,s)=d(x,t) 
\quad \mbox{for all $(x,t) \in \mathbf{R}^n \times (0,T)$} , \]

\item\label{itm:main1.5} $u^\theta (\cdot ,t)$ converges to $d(\cdot ,t)$ locally uniformly in $\mathbf{R}^n$ as $\theta \rightarrow +\infty$ for all $t\in (0,T)$,

	\item\label{itm:main2}  if in addition $d(x,t)$ is continuous in $\mathbf{R}^n\times (0,T)$, then
\[
u^\theta \text{ converges to } d \text{ locally uniformly in } \mathbf{R}^n\times (0,T)\text{ as } \theta \rightarrow +\infty.
\]
 \end{enumerate} 
\end{theo}
In general, if the signed distance function $d$ is discontinuous, we cannot expect that the continuous functions $u^\theta$ will converge to $d$ locally uniformly. The following example shows that the signed distance function can be discontinuous when points of the zero level set disappear instantaneously.
We will denote by $B_r(x)$ the open ball of radius $r>0$ centered at $x$. Its closure is $\overline{B_r(x)}$. Also, $\langle \cdot ,\cdot \rangle$ stands for the standard Euclidean inner product.
\begin{ex}[A single discontinuity]\label{exmp:sdisc}
We study \eqref{eq:init} with
\begin{equation}\label{eq:exmpH_1}
H_1(x,p)=c(x)|p|,
\end{equation}
where $c\in \mathrm{Lip}(\mathbf{R}^n)$ is bounded and non-negative.
Since $H_1$ is written as $H_1(x,p)=\displaystyle\max_{a \in \overline{B_1(0)}} \langle c(x)a,p \rangle$,
the viscosity solution $w$ of \eqref{eq:init}
has a representation formula as a value function 
of the associated optimal control problem (\cite[Section 10]{Evans_Book}),
which is of the form
\begin{equation} 
w(x,t)=\sup_{\alpha \in \mathcal{A}} 
u_0(X^{\alpha}(t)).
\label{eq:rep_opt}
\end{equation}
Here $\mathcal{A}:=\{ \alpha:[0,T) \to \overline{B_1(0)}, \ \mbox{measurable} \}$
and $X^{\alpha}:[0,T) \to \mathbf{R}^n$ is the solution of the state equation
\[ (X^{\alpha})'(s)=c(X^{\alpha}(s)) \alpha (s) \quad \mbox{in} \ (0,T),
\quad X^{\alpha}(0)=x. \]
Each element $\alpha \in \mathcal{A}$ is called a control.

We now consider the case where $c(x)=1$.
This describes a phenomenon where the interface expands at a uniform speed 1.
In this case the optimal control forces the state $X^{\alpha}(\cdot)$
to move towards the maximum point of $u_0$ in $\overline{B_t(x)}$,
and hence
\begin{equation} 
w(x,t)=\max_{|x-y|\leq t} u_0(y).
\label{eq:rep_opt_max}
\end{equation}
Take the initial datum as $u_0(x)=\max\{ (1-|x-2|)_+, \ (1-|x+2|)_+ \}$.
The formula \eqref{eq:rep_opt_max} now implies
\[ w(x,t)=\min\{ \max\{ (t+1-|x-2|)_+, \ (t+1-|x+2|)_+ \}, 1 \}. \]
 To see this we notice that
\[
\max_{|x-y|\leq t}(1-|y\pm 2|)_+=\min\{(t+1-|x\pm 2|)_+,1\},
\]
then using the formula \eqref{eq:rep_opt_max} and after changing the order of the maxima, we calculate
\begin{align*}
w(x,t) &= \max\{ \min\{(t+1-|x- 2|)_+,1\},\min\{(t+1-|x+ 2|)_+,1\}\\
&= \min\{ \max\{ (t+1-|x-2|)_+, \ (t+1-|x+2|)_+ \}, 1 \} .
\end{align*}
Here we have used the relation
$\max\{ \min\{ a_1, \ b \} \ \min\{ a_2, \ b \} \}
=\min\{ \max\{ a_1, \ a_2 \}, \ b \}$
for $a_1,a_2,b \in \mathbf{R}$.
We therefore have
\[ \{ w=0 \}=\begin{cases} 
\{ |x|\geq t+3 \} \cup \{ |x|\leq 1-t \} & \mbox{if} \ t\leq 1, \\ \{ |x|\geq t+3 \} & \mbox{if} \ t>1 \end{cases} \]
and 
\[ d(x,t)=\begin{cases} 
\max\{ (t+1-|x-2|)_+, \ (t+1-|x+2|)_+ \} & \mbox{if} \ t\leq 1, \\ (t+3-|x|)_+ & \mbox{if} \ t>1. \end{cases} \]
See Figure \ref{fig:d4_discd}. Thus $d$ is discontinuous on $\ell :=\{ (x,1) \ | \ -2<|x|<2 \}$;
more precisely, 
$d$ is not upper semicontinuous but lower semicontinuous on $\ell$.


\begin{figure}[htbp]
\begin{center}
\includegraphics{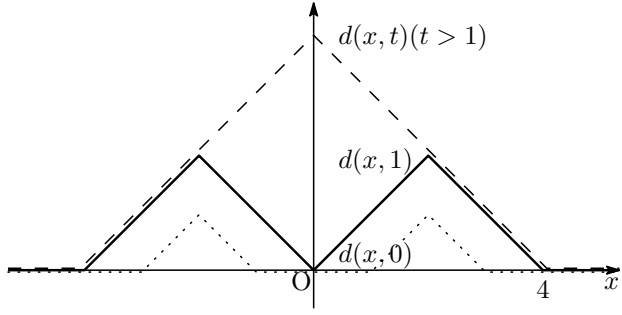}
\end{center}
\caption{The graph of $d$.}
\label{fig:d4_discd}
\end{figure}

\end{ex}

If $h$ satisfies \eqref{eq:hprop3} we can still estimate from one side the limit with the distance function. More precisely we have the following theorem.
\begin{theo}\label{Theo1.1}
Assume \eqref{eq:theor1_assumpt} and \eqref{eq:hprop3}. Then
\begin{align*}
&& d(x,t)&\leq \sup_{\theta >0} u^{\theta}(x,t)<+\infty
&\mbox{for all} \ x \in D^+_t, \\
&& d(x,t)&\geq \inf_{\theta >0} u^{\theta}(x,t)>-\infty
&\mbox{for all} \ x \in D^-_t. 
\end{align*}
\end{theo}

For the next result we define
\begin{equation}\label{eq:combined_hamilt}
H_{12}(x,t,\tau ,r,p):=\begin{cases}
H_1(x,\frac{t}{1+\frac{k_2}{k_1}},p)
& \mbox{if} \ (i-1) < \tau \leq  (i-1) + \frac{k_1 \Delta t}{\varepsilon}, \\
H_2(r,p)
& \mbox{if} \ (i-1)+ \frac{k_1 \Delta t}{\varepsilon} < \tau \leq  i
\end{cases}
\end{equation}
for $k_1,k_2>0$, $\Delta t>0$, $\varepsilon =(k_1+k_2)\Delta t$ and $i=1,...,\lceil \frac{T}{\varepsilon}\rceil$.
By definition $H_{12}$ is 1-periodic in $\tau $, and
in general it is discontinuous in $\tau $.
In summary, we are led to the following equation:
\begin{equation}\label{eq:HJe}
 u^\varepsilon _t  = H_{12}\left(  x,t,\frac{t}{\varepsilon},u^\varepsilon ,\nabla u^\varepsilon \right)\quad \text{ in } \mathbf{R}^n\times (0,T).
\end{equation}
\begin{rem}\label{rem:iterative_sol}\label{Rem:iterative_sol}
 A solution of the problem \eqref{eq:HJe}, \eqref{eq:initin} can be constructed by solving \eqref{eq:HJe} in the intervals $[\varepsilon (i-1),\varepsilon (i-1)+k_1 \Delta t),\, [\varepsilon (i-1)+k_1 \Delta t,\varepsilon i)$, $i=1,...,\lceil \frac{T}{\varepsilon}\rceil$, iteratively, using as initial condition at each interval, the final value of the solution defined in the previous interval. We call this solution an {\em iterative} solution.
\end{rem}
Let $\theta=k_2/k_1$. We define 
\begin{equation*}
\bar{H}(x,t,r,p)=\frac{1}{1+\theta}\left( H_1\left( x,\frac{t}{1+\theta},p \right)+\theta H_2(r,p)\right) 
\end{equation*}
and consider the equation
\begin{equation}\label{eq:mean_hmilt_eq}
\bar{u}_t ^\theta =\bar{H}(x,t,\bar{u}^\theta ,\nabla \bar{u}^\theta )\quad \text{ in } \mathbf{R}^n \times (0,T).
\end{equation}
\begin{theo}[Homogenization]\label{Theor2}
Assume \eqref{eq:theor1_assumpt}. Let $\bar{u}^\theta$ and $u^\varepsilon$ be, respectively, the solution of \eqref{eq:mean_hmilt_eq}, \eqref{eq:initin} and the iterative solution of \eqref{eq:HJe}, \eqref{eq:initin}. Then $u^\varepsilon$ converges to $\bar{u}^\theta$ locally uniformly in $\mathbf{R}^n \times [0,T)$.
\end{theo}


\begin{rem}\label{rem2}
If we set $u^\theta (x,t)=\bar{u}^\theta (x,(1+\theta )t)$ in Theorem \ref{Theor2}, then $u^\theta$ solves the equation \eqref{eq6} and satisfies the initial data \eqref{eq:initin}.
\end{rem}
\begin{rem}
All of our main theorems have the same assumptions 
on $u_0$ and $H_1$. 
For this reason, we will assume \eqref{eq:theor1_assumpt}
in the rest of the paper except Subsection \ref{subsect:preservation}, 
where $u_0$ will be generalized.
For the function $h$ in \eqref{eq1.0*} we will differentiate the assumptions 
\eqref{eq:hprop2} and \eqref{eq:hprop3}. 
Finally we will state clearly whether or not 
the distance function $d$ is continuous.
\end{rem}

\subsection{Theorems from the literature}\label{section:prepar}

In this subsection we will present a comparison principle for general equations of the form
 \begin{equation}\label{eq5}
 u_t  =F(x,t,u,\nabla u) \quad \text{ in } \mathbf{R}^n \times (0,T).
\end{equation}

Let us introduce a notion of viscosity solutions. For this purpose, we first define semicontinuous envelopes of functions. Let $K\subset \mathbf{R}^n$. For a function $f:K\rightarrow \mathbf{R}$ we denote the upper and lower semicontinuous envelopes by $f^*$ and $f_*:\overline{K}\rightarrow \mathbf{R}\cup \{\pm \infty \}$ respectively, which are as follows:
\[
f^* (z):= \limsup_{y\rightarrow z} f(y) =\lim_{\delta\rightarrow 0} \sup \{ f(y)\mid y\in B_\delta (z)\cap K\} ,
\]
\[
f_* (z):= \liminf_{y\rightarrow z} f(y) =\lim_{\delta\rightarrow 0} \inf \{ f(y)\mid y\in B_\delta (z)\cap K\}.
\]
\begin{defi}[Viscosity Solution]\label{defi1}
We say that $u:\mathbf{R}^n\times [0,T)\rightarrow \mathbf{R}$ is a {\em viscosity subsolution} (resp. a {\em supersolution}) of \eqref{eq5} if $u^* <+\infty $(resp. $u_* >-\infty$) and if
 \[
 \phi _t\leq F^* (x_0,t_0,\phi ,\nabla\phi )\quad (\text{resp. } \phi _t\geq F_* (x_0,t_0,\phi ,\nabla\phi )) \quad \text{ at } P_0=(x_0,t_0)
 \]
whenever
\begin{equation}\label{eq5.5}
\left\{
\begin{aligned}
& u^* \leq \phi &\text{ on }& B_{r_0}(P_0)\\
& u^*=\phi  & \text{ at } & P_0
\end{aligned}
\right.
\quad \left( \text{resp. } \left\{
\begin{aligned}
& u_* \geq \phi &\text{ on }& B_{r_0}(P_0)\\
& u_*=\phi  & \text{ at } & P_0
\end{aligned}
\right. \right)
\end{equation}
for $\phi \in C^1(\mathbf{R}^n \times (0,T))$, $P_0\in \mathbf{R}^n \times (0,T)$ and $r>0$ such that $B_r(P_0)\subset \mathbf{R}^n \times (0,T)$.
\end{defi}
Since we already use the notation $D^{\pm}$, we are going to use the symbol $\mathcal{J}^{\pm}$ for the {\em subdifferential} respectively for the {\em superdifferential} of a function.
More precisely for a function $u : \mathbf{R}^n \times (0,T) \to \mathbf{R}$
we define a {\em superdifferential} $\mathcal{J}^+ u(z,s)$ 
of $u$ at $(z,s) \in \mathbf{R}^n \times (0,T)$ by
\begin{equation} 
\mathcal{J}^+ u(z,s):=
\left\{ ( p,\tau )\in \mathbf{R}^n \times \mathbf{R} \ \left| 
\begin{array}{c} 
\exists \phi \in C^1 (\mathbf{R}^n \times (0,T)) \ \mbox{such that} \\ 
(p,\tau)=(\nabla \phi,\partial_t \phi)(z,s) \ \mbox{and} \\
\displaystyle\max_{\mathbf{R}^n \times (0,T)}(u-\phi)=(u-\phi)(z,s) 
\end{array} \right. \right\}. 
\label{eq:DefofSupDiff}
\end{equation}
A {\em subdifferential} $\mathcal{J}^- u(z,s)$ is defined 
by replacing ``$\max$" by ``$\min$" in \eqref{eq:DefofSupDiff}. Equivalently, we say that a function
$u : \mathbf{R}^n \times (0,T) \to \mathbf{R}$
is a {\em viscosity subsolution} (resp. {\em viscosity supersolution})
of \eqref{eq5} if 
\[ \tau \leq F^*(z,s,u^*(z,s),p) 
 \ \mbox{(resp. $\tau \geq F_*(z,s,u_*(z,s),p) $ )} \]
for all $(z,s) \in \mathbf{R}^n \times (0,T)$
and $(p,\tau) \in \mathcal{J}^+ u^*(z,s)$
(resp. $(p,\tau) \in \mathcal{J}^- u_*(z,s)$).\par
In order to guarantee the well-posedness of the problem \eqref{eq5} and \eqref{eq:initin}, the following assumptions are usually imposed on the function $F:\mathbf{R}^n\times [0,T]\times \mathbf{R}\times \mathbf{R}^n\rightarrow \mathbf{R}$.

\begin{enumerate}[label=(F\arabic*) ,ref=F\arabic*]
	\item\label{itmF1} $F\in C(\mathbf{R}^n \times [0,T]\times \mathbf{R}\times \mathbf{R}^n)$,
	\item\label{itmF2} There is an $a_0 \in \mathbf{R}$ such that $r\mapsto F(x,t,r,p)-a_0 r$ is non-increasing on $\mathbf{R}$ for all $(x,t,p)\in \mathbf{R}^n \times [0,T]\times\mathbf{R}^n$,
	\item\label{itmF3} For $R\geq 0$ there is a modulus $\omega _R$ such that
	\[
	|F(x,t,r,p)-F(x,t,r,q)|\leq \omega _R(|p-q|)
	\]
	for all $(x,t,r,p,q)\in \mathbf{R}^n \times [0,T]\times \mathbf{R}\times \mathbf{R}^n\times \mathbf{R}^n$, with $|p|,|q|\leq R$,
	\item\label{itmF4} There is a modulus $\omega$ such that
	\[
	|F(x,t,r,p)-F(y,t,r,p)|\leq \omega (|x-y|(1+|p|))
	\]
	for all $(x,y,t,r,p)\in \mathbf{R}^n\times\mathbf{R}^n \times [0,T]\times \mathbf{R}\times \mathbf{R}^n$.
	 \end{enumerate}

For the convenience of the reader we will state the comparison principle and sketch its proof for the problem \eqref{eq5} and \eqref{eq:initin}. For a detailed proof, see \cite[Theorem 4.1]{MR1119185}.

\begin{pro}[Comparison pinciple]\label{prop3}
Assume that $F$ satisfies \eqref{itmF1}-\eqref{itmF4}. Let $u,v$ be a viscosity subsolution and supersolution respectively of \eqref{eq5} and assume that they satisfy
\begin{enumerate}[label=(A\arabic*) ,ref=A\arabic*]
	\item\label{itmA1} $u^*(x,0)\leq v_*(x,0)$ for all $x\in \mathbf{R}^n$,
	\item\label{eq8} there is a constant $K>0$ such that we have on $\mathbf{R}^n\times (0,T)$
	\[
	u(x,t)\leq  K(1+|x|),\quad v(x,t)\geq  -K(1+|x|),
	\]
	\item\label{itmA3} there is a constant $\tilde{K}>0$ such that for $x,y\in\mathbf{R}^n$ we have
	\[
	u^* (x,0)-v_*(y,0)\leq \tilde{K}|x-y|.
	\]
	 \end{enumerate}
Then
\[
u^*\leq v_*\, \text{ in } \mathbf{R}^n\times [0,T).
\]
\end{pro}
\begin{proof}
Since the comparison principal is more or less classical we will only give a sketch of the proof.\par 
1.
We may suppose without loss of generality that $u,v$ are upper, respectively, lower semicontinuous. As usual we set
\[
\tilde{u}(\cdot ,t)=e^{-a t}u(\cdot ,t),\, \tilde{v}(\cdot ,t)=e^{-a t}v(\cdot ,t),
\]
where $a>a_0$ and $a_0$ is given by \eqref{itmF2}. Using the notation $u,v$ instead of $\tilde{u},\tilde{v}$, we have that $u,v$ are sub- and supersolutions of the equation
\begin{equation*}
u_t+(a-a_0)u=\tilde{F}(x,t,u,\nabla u),
\end{equation*}
where $\tilde{F}(x,t,r,p)$ is non-increasing in $r$ and satisfies the conditions \eqref{itmF1}-\eqref{itmF4}. As before we denote by $F$ the new $\tilde{F}$. \par 
2.
Suppose that
\[
M=\sup_{(x,t)\in \mathbf{R}^n \times [0,T)}(u(x,t)-v(x,t))>0.
\]
We now make the usual doubling of variables trick and define for $\varepsilon ,\eta ,\alpha >0$ 
\[
\Phi _\sigma (x,t,y)=u(x,t)-v(y,t)-\frac{|x-y|^2}{\varepsilon}-\frac{\eta}{T-t}-\alpha (|x|^2+|y|^2)
\]
and
\[
M_\sigma =\sup_{x,y\in \mathbf{R}^n,\, t\in [0,T)}\Phi _\sigma (x,y,t),
\]
where $\sigma =(\varepsilon ,\eta ,\alpha )$. As usual we have $0<M_\sigma <+\infty$. In order to proceed we need a priori bounds on the maximum $M_\sigma$ and to do that we need to be able to control the difference $u(x,t)-v(y,t)$ by the modulus $|x-y|$. One can show (using a doubling of variables trick, see for example \cite[Proposition $2.3'$]{MR1119185}) that there is a constant $C_T>0$ such that
\[
u(x,t)-v(y,t)\leq C_T(1+|x-y|).
\]
Using this estimate we can show that there is $C>0$ such that
\[
\alpha |x|,\alpha |y|\leq \sqrt{\alpha C}.
\]
The above estimate together with \eqref{itmA3} enables us to find $x_\sigma ,y_\sigma \in \mathbf{R}^n$ and $t_\sigma \in (0,T)$ such that $\Phi _\sigma (x_\sigma ,y_\sigma ,t_\sigma )=M_\sigma$.\par 
For the term $|x_\sigma -y_\sigma |^2 /\varepsilon$, we will need a more refined estimate than the classical one, namely we need
\begin{equation}\label{eq:refined_estimate}
\lim_{\varepsilon\rightarrow 0}\left( \limsup_{\eta ,\alpha\rightarrow 0 }\frac{|x_\sigma -y_\sigma |^2}{\varepsilon }\right)=0.
\end{equation}
A proof of a similar estimate can be found in \cite[Proposition 4.4]{MR1119185}.\par 
3.
Doubling the variables again in time or using a similar argument as in \cite[Lemma 2]{Crandall_Lions}, we have for $p_\varepsilon =2(x_\sigma -y_\sigma )/\varepsilon$ 
\begin{align*}
\frac{\eta}{T^2}+(a-a_0)(u(x_\sigma ,t_\sigma )-& v(y_\sigma ,t_\sigma ))\\
&\leq F(x_\sigma ,t_\sigma ,u(x_\sigma ,t_\sigma ) ,p_\varepsilon +2\alpha x_\sigma )-F(y_\sigma ,t_\sigma ,v(y_\sigma ,t_\sigma ),p_\varepsilon -2\alpha y_\sigma  ).
\end{align*}
As in \cite[Proposition 2.4]{MR1119185} there is a $\delta >0$ independent of $\sigma$, such that $u(x_\sigma ,t_\sigma )-v(y_\sigma ,t_\sigma )>\delta$.
Using properties of $F$, \eqref{itmF3} and \eqref{itmF4}, one gets
\[
\frac{\eta}{T^2}+(a-a_0)\delta \leq \omega (|x_\sigma -y_\sigma |(1+|p_\varepsilon +2\alpha x_\sigma |))+\omega _R(2\alpha (|x_\sigma |+|y_\sigma |)),
\]
where $R=o(1/\sqrt{\varepsilon})+o(\sqrt{\alpha})$ as $\varepsilon ,\alpha \rightarrow 0$. Using \eqref{eq:refined_estimate}, we can take the limit first as $\alpha ,\eta \rightarrow 0$ and then as $\varepsilon\rightarrow 0$ to get a contradiction.
\end{proof}
Combining Proposition \ref{prop3} with Perron's method we get the following theorem.
\begin{theo}[Existence/Uniqueness]\label{the2}
Assume \eqref{eq:theor1_assumpt}. Then for all $\theta >0$, there exists a unique solution $u=u^\theta\in C(\mathbf{R}^n\times [0,T))$ of the problem \eqref{eq6} and \eqref{eq:initin} with
\[
u^{\mathit{low}}\leq u\leq u^\mathit{up} \quad \text{ in } \mathbf{R}^n \times [0,T),
\]
where $u^\mathit{up}(x,t)=u_0(x)+Kt,\, u^\mathit{low}(x,t)=u_0(x)-Kt$, for some $K>0$, are a viscosity supersolution and subsolution respectively of the same problem.
\end{theo}
\begin{proof}
We will only show that $u_0\pm Kt$ are a subsolution and a supersolution of \eqref{eq6} for some $K>0$, large enough depending on $\theta$, since for the rest of the proof we can use a Perron's argument, see for example \cite{Ishii.Perron}. We first suppose that $u_0$ is smooth. Then by the Lipschitz continuity of $u_0$ we have $|\nabla u_0|\leq L_0$. By assumptions \eqref{itm5} and \eqref{itm7}, there is a constant $C>0$ such that $|H_1(x,t,\nabla u_0)|\leq CL_0$. Also, since $h$ is continuous and $\beta$ is bounded by \eqref{eq1.1*}, there is a constant $M>0$ such that
\[
|H_2(u_0,\nabla u_0)|=\theta |\beta (u_0)h(\nabla u_0)|\leq \theta \max_{|p|\leq L_0}|h(p)|M.
\]
 Finally, if we choose $K>0$ such that $K \geq CL_0+\theta \displaystyle \max_{|p|\leq L_0}|h(p)|M$, we get the desired result. For the case where $u_0$ is not smooth, we use the same argument for elements of the super- and subdifferential of $u_0$.
\end{proof}
The following proposition is proved in Appendix \ref{appendix}. Define 
\[ L(t):=\max \{ L_0,1\} e^{\int_0^t D(s)\, ds }. \]
\begin{pro}[Lipschitz continuity of solutions]\label{prop2}
Under the assumptions of Proposition \ref{prop3}, with $H_1$ satisfying assumption \eqref{itm6'} instead of \eqref{itm6}, the solution $u$ of the problem \eqref{eq6} and \eqref{eq:initin} satisfies
\[
|u(x,t)-u(y,t)|\leq L(t)|x-y|\quad \text{  for all } x,y\in \mathbf{R}^n\text{ and } t\in [0,T).
\]
\end{pro}
\begin{rem}
The Lipschitz continuity of the solution of \eqref{eq:init} will be used in the next section
to show that the solution $u^{\theta}$ of \eqref{eq6} gives the same zero level set as \eqref{eq:init}
and that there exist barrier functions of $u^{\theta}$  independent of $\theta$.
There, the Lipschitz constant is allowed to depend on the terminal time $T$.
It is well-known that, if $H_1$ is {\em coercive}, i.e.,
\[ H_1(x,t,p) \to \infty \quad 
\mbox{as $|p| \to \infty$ uniformly in $(x,t)$}, \]
then the solution is Lipschitz continuous and
its Lipschitz constant does not depend on $T$.
See, e.g., \cite{Barles.1990}.
Since such independence of $T$ is not needed for our study,
we do not require $H_1$ to be coercive in this paper.
\end{rem}
One important property of geometric equations \eqref{eq:initeq} is the
invariance under the change of dependent variables.
This invariance property as well as the comparison principle
play a crucial role for the proof of uniqueness of evolutions.

\begin{theo}[Invariance]\label{thm:InvLSM}
Let $\theta : \mathbf{R} \to \mathbf{R}$ 
be a nondecreasing and upper semicontinuous (resp. lower semicontinuous) function.
If $w$ is a viscosity subsolution (resp. supersolution) of \eqref{eq:initeq},
then $\theta \circ w$ is a viscosity subsolution (resp. supersolution) of \eqref{eq:initeq}.
\end{theo}
See \cite[Theorem 4.2.1]{Giga.See} for the proof.


\begin{rem}\label{rem:inv_char}
As a simple consequence of the invariance property,
we see that, when $w$ is a solution of \eqref{eq:initeq},
the characteristic function on $D^+_t$(see \eqref{zeronot1}) defined as 
\[ \chi_{D^+_t}(x)=\begin{cases} 1 & \text{ if } x\in D_t ^+, \\
0 & \text{ if } x\notin D_t ^+ \end{cases} \]
is a supersolution of \eqref{eq:initeq}
since it is written as $\chi_{D^+_t}(x)=\chi_{(0,\infty)}(x) \circ w$.
Similarly, $\chi_{D^+_t \cup \Gamma_t}(x)$ is a subsolution of \eqref{eq:initeq}.
\end{rem}

It is known that the evolution of the interface 
$\{ \Gamma_t \}_{t \in (0,T)}$ associated with \eqref{eq:initeq} 
is independent of a choice of the initial data $u_0$.
In other words, if the zero levels of initial data are the same,
then those of the solutions are also the same.
See \cite[Section 4.2.3 and 4.2.4]{Giga.See}
for the detailed statement and its proof.
\section{Main tools}\label{section:main_tools}
\subsection{Preservation of the zero level set}\label{subsect:preservation}

We believe that the preservation of the zero level set is by itself a useful result. For this reason we present it in a more general framework than the one we are going to apply it for the proof of our main results.

We study a general equation of the form
\begin{equation}
u_t=H_1(x,t,\nabla u)+\beta(u)G(x,\nabla u)
\quad \mbox{in} \ \mathbf{R}^n \times (0,T).
\label{eq:GenHGeq}
\end{equation}
The function $H_1$ satisfies \eqref{itm4}-\eqref{itm7}. For the function $\beta$ we assume that \eqref{eq1.1*} is true, and for 
$G:\mathbf{R}^n \times \mathbf{R}^n \to \mathbf{R}$ we assume

\begin{enumerate}[label=(G) ,ref=G]
	\item\label{itm1*} $G$ satisfies \eqref{itmF3}, \eqref{itmF4} and is bounded from above in $\mathbf{R}^n \times \mathbf{R}^n$.
 \end{enumerate}

Under these assumptions the comparison principle holds for solutions of \eqref{eq:GenHGeq}. Indeed, the continuity assumptions on $H_1$, $G$ and $\beta$ imply 
that the function $F(x,r,p):=H_1(x,t,p)-\beta(r) G(x,p)$ satisfies \eqref{itmF1}, \eqref{itmF3}, \eqref{itmF4}
while \eqref{itmF2} is fulfilled with 
$\gamma =\Lb (\sup_{\mathbf{R}^n \times \mathbf{R}^n} G)$.

To guarantee that solutions of \eqref{eq:GenHGeq} preserve the original zero level set,
two kinds of sufficient conditions on $G$ are made in our theorem.
One is boundedness of $G$ from below,
which, unfortunately, excludes the typical case $G(x,p)=1-|p|$.
The other condition needs only local boundedness of $G$ from below near $p=0$
but requires solutions of \eqref{eq:init} to be Lipschitz continuous, which is not true in general if the initial data $u_0$ is just uniformly continuous.\par 
In the first author's dissertation \cite{Nao-thesis},
Theorem \ref{thm:PresZLS} \eqref{itm1**} is established,
but we give its proof here not only for the reader's convenience 
but also in order to show connection with the proof of \eqref{itm2**}.

\begin{theo}[Preservation of the zero level set]\label{thm:PresZLS}
Let $w$ and $u$ be, respectively,
the viscosity solution of \eqref{eq:init} and \eqref{eq:GenHGeq}, \eqref{eq:initin}
with a uniformly continuous $u_0$.
Assume either \eqref{itm1**} or \eqref{itm2**} below:
\begin{enumerate}[label=(\roman*) ,ref=\roman*]
	\item\label{itm1**}$G$ is bounded from below in $\mathbf{R}^n \times \mathbf{R}^n$.
	\item\label{itm2**}$G$ is bounded from below in $\mathbf{R}^n \times \overline{B_{\rho}(0)}$ for some $\rho\in (0,1]$,
and $\mathrm{Lip}_x[w]< \infty$.
 \end{enumerate}
Then we have $\Gamma_t=\{ u(\cdot ,t)=0\}$ 
and $D^{\pm}_t=\{ \pm u(\cdot ,t)>0\}$ for all $t \in (0,T)$, where $D^{\pm}$ and $\Gamma _t$ are defined in \eqref{zeronot1} and \eqref{zeronot2} respectively.
\end{theo}
\begin{rem}
Assume that $G$ is independent of $x$ and continuous. Then \eqref{itm1**} is true if $G$ satisfies \eqref{eq:hprop3}, while \eqref{itm2**} is true if $G$ satisfies \eqref{eq:hprop2}.
\end{rem}
\begin{proof}
Assume that \eqref{itm1**} is true.\par 
1.
Set $G^\star =\max \{  \sup_{\mathbf{R}^n \times \mathbf{R}^n} G, \ 0 \}$
and $G_\star:=\max \{ -\inf_{\mathbf{R}^n \times \mathbf{R}^n} G, \ 0 \}$.
We define
\[ v^\star(x,t):=
\begin{cases}
e^{\Lb G^\star t} w(x,t)
& \mbox{if} \ w(x,t) \geq 0, \\
e^{-\Lb G_\star t} w(x,t)
& \mbox{if} \ w(x,t) < 0
\end{cases} \]
and 
\[ v_\star(x,t):=
\begin{cases}
e^{-\Lb G_\star t} w(x,t)
& \mbox{if} \ w(x,t) \geq 0, \\
e^{\Lb G^\star t} w(x,t)
& \mbox{if} \ w(x,t) < 0
\end{cases} \]
for $(x,t) \in \mathbf{R}^n \times [0,T)$,
where $\Lb$ is the Lipschitz constant of $\beta$ appearing in \eqref{eq1.1*}.
We claim that $v^\star$ and $v_\star$ are, respectively,
a viscosity supersolution and subsolution of \eqref{eq:GenHGeq}.\par
2.
We shall show that $v^\star$ is a supersolution.
If $w$ is smooth and $w(x,t)>0$, we compute
\begin{align*}
 v_t ^\star -H_1(x,t,\nabla v^\star)
&=\Lb G^\star v^\star + e^{\Lb G^\star t} w_t -H_1(x,t,e^{\Lb G^\star t} \nabla w) \\
&=\Lb G^\star v^\star + e^{\Lb G^\star t} \{ w_t -H_1(x,t, \nabla w) \} \\
&\geq \Lb G^\star v^\star +0 \\
&\geq \beta (v^\star ) G(x,\nabla v^\star ),
\end{align*}
which implies that $v^\star$ is a supersolution of \eqref{eq:GenHGeq}.
In the general case where $w$ is not necessarily smooth,
taking an element of the subdifferential of $w$,
we see that $v^\star$ is a viscosity supersolution of \eqref{eq:GenHGeq}.
Similar arguments apply to the case when $w(x,t)<0$,
so that $v^\star $ is a supersolution in $\{ w>0 \} \cup \{ w<0 \}$.
It remains to prove that $v^\star $ is a supersolution of \eqref{eq:GenHGeq} on $\{ w=0 \}$.\par 
Let $(z,s) \in \mathbf{R}^n \times (0,T)$ be a point such that $w(z,s)=0$,
and take $(p,\tau) \in \mathcal{J}^- v^\star (z,s)$.
Our goal is to derive 
$$\tau \geq H_1(z,s,p)$$
 since $\beta(v^\star (z,s))=0$.
To do this, we consider a characteristic function $g(x,t)=\chi_{D_t}(x)$.
We have $v^\star(z,s)=g(z,s)=0$ and $v^\star \leq g$ near $(z,s)$,
and thus $(p,\tau) \in \mathcal{J}^- g (z,s)$.
Since $g$ is a supersolution of \eqref{eq:initeq} by Remark \ref{rem:inv_char},
we have $\tau \geq H_1(z,s,p)$, which is the desired inequality.
Summarizing the above arguments, we conclude that $v^\star$ is a supersolution of \eqref{eq:GenHGeq}.
In the same manner we are able to prove that 
$v_\star$ is a subsolution of \eqref{eq:GenHGeq}.\par 
3.
Since $v^{\star}(x,0)=v_\star(x,0)=u_0(x)$ for all $x \in \mathbf{R}^n$,
the comparison principle (Proposition \ref{prop3}) yields 
\[
v_\star (x,t) \leq u(x,t) \leq v^\star(x,t) \text{ for all } (x,t) \in \mathbf{R}^n \times (0,T).
\]
In particular, we have
\[
 \{ v_\star (\cdot ,t)>0 \}\subset \{u(\cdot ,t)>0 \} \subset \{v^\star (\cdot ,t)>0\}.
 \]
Since $\{v_\star (\cdot ,t)>0\}=\{v^\star (\cdot ,t)>0\}=D^+_t$ by the definition of $v_\star$ and $ v^\star$,
we conclude that $D^+_t=\{u(\cdot ,t)>0\}$.
Similarly, we obtain $D^-_t=\{u(\cdot ,t)<0\}$,
and hence $\Gamma_t=\{u(\cdot ,t)=0\}$.\par 
Assume that \eqref{itm2**} is true.\par 
1.
Set $G_{\star\rho}:=\max \{ -\inf_{\mathbf{R}^n \times \overline{B_{\rho}(0)}} G, \ 0 \}$
and $m:=\max\{ \mathrm{Lip}_x [w], \ 1 \}$.
Instead of $v_\star$ and $ v^\star$ defined in Step 1 of \eqref{itm1**}, we consider the functions
\[ \tilde{v}^\star(x,t):=
\begin{cases}
e^{\Lb G^\star t} w(x,t)
& \mbox{if} \ w(x,t) \geq 0, \\
(\rho / m ) e^{-\Lb G_{\star\rho} t} w(x,t)
& \mbox{if} \ w(x,t) < 0
\end{cases} \]
and 
\[ \tilde{v}_\star(x,t):=
\begin{cases}
(\rho / m ) e^{-\Lb G_{\star\rho} t} w(x,t)
& \mbox{if} \ w(x,t) \geq 0, \\
e^{\Lb G^\star t} w(x,t)
& \mbox{if} \ w(x,t) < 0.
\end{cases} \]
Then $\tilde{v}^\star$ and $\tilde{v}_\star$ are
a viscosity supersolution and subsolution of \eqref{eq:GenHGeq} respectively.\par 
2.
We shall prove that $\tilde{v}_\star$ is a subsolution in $\{ w>0 \}$.
If $w$ is smooth, we have
$$|\nabla \tilde{v}_\star|=(\rho / m ) e^{-\Lb G_{\star\rho} t} |\nabla w|\leq \rho,$$
which implies that
$G(x,\nabla \tilde{v}_\star )\geq -G_{\star\rho}$.
Similarly to Step 2 of \eqref{itm1**},
we observe
\begin{align*}
 (\tilde{v}_\star )_t -H_1(x,t,\nabla \tilde{v}_\star)
&=-\Lb G_{\star\rho} \tilde{v}_\star
  +(\rho / m ) e^{-\Lb G_{\star\rho} t}\{ w_t -H_1(x,t, \nabla w) \} \\
&\leq -\Lb G_{\star\rho} \tilde{v}_\star +0 \\
&\leq \beta (\tilde{v}_\star ) G(x,\nabla \tilde{v}_\star),
\end{align*}
i.e., $\tilde{v}_\star$ is a subsolution.
The rest of the proof runs as before.
\end{proof}
As an immediate consequence of Theorem \ref{thm:PresZLS},
it follows that the evolution which is given as the zero level set of 
the solution of the non-geometric equation \eqref{eq:GenHGeq}
does not depend on the choice of its initial data.

\subsection{Barrier functions}\label{Sect:Barrier}
Throughout this subsection we will assume \eqref{eq:theor1_assumpt}.
Thanks to Theorem \ref{thm:PresZLS}, for a general $h$ satisfying \eqref{hprop1} and either one of the assumptions \eqref{eq:hprop2} or \eqref{eq:hprop3},
the solution $u^{\theta}$ of \eqref{eq6} and \eqref{eq:initin}
gives the same zero level set as $w$,  i.e., we have
$\Gamma_t=\{u^{\theta}=0\}$ and $D^{\pm}_t=\{\pm u^{\theta}>0\}$
for all $t\in (0,T)$.
In order to study the behaviour of $u^\theta$ as $\theta\rightarrow \infty$ and a relation between $\Gamma _t$ and the zero level set of the limit of $u^\theta$, we will construct barrier functions independent of $\theta$.
More precisely we construct an upper barrier $f^\star$ and a lower barrier $f_\star$ such that
\begin{align*}
& f_\star \leq u^{\theta} \leq f^\star ,\,\Gamma_t=\{ f^\star =0\}=\{f_\star =0\},\\
& \quad\quad D^+=\{f^\star >0\}=\{f_\star >0\},\\
&\quad\quad D^-=\{f^\star <0\}=\{f_\star <0\}.
\end{align*}
In this subsection we often use the fact that,
if $u$ is a supersolution (resp. subsolution) of \eqref{eq6} in $D^+$,
and if $\Gamma =\{u=0\}$ and $D^{\pm}=\{\pm u>0\}$,
then $u_+$ is a supersolution (resp. subsolution) of \eqref{eq6} in $\mathbf{R}^n \times (0,T)$.
This follows from Remark \ref{rem:inv_char}.
Indeed, if $(p,\tau) \in \mathcal{J}^- u_+ (z,s)$ and $u_+(z,s)=0$,
then we have $(p,\tau) \in \mathcal{J}^- \chi_{D^+} (z,s)$ and 
this yields the desired viscosity inequality
since the characteristic function is a supersolution of \eqref{eq:initeq} by Remark \ref{rem:inv_char}.
The proof for a subsolution is similar.

We first show that the solutions $u^{\theta}$ are monotone 
with respect to $\theta$ when $h$ is nonnegative.
This gives a lower barrier in $D^+$ and an upper barrier in $D^-$ 
in the case of \eqref{eq:hprop3}.

\begin{pro}[Monotonicity]\label{prop:mono}
Assume that $h \geq 0$.
Let $0<\theta_1< \theta_2$ and 
$u^{\theta_1}$ and $u^{\theta_2}$ be, respectively, 
the viscosity solution of \eqref{eq6}
with $\theta=\theta_1$ and $\theta_2$.
Then 
\begin{align}
&& u^{\theta_1}(x,t)&\leq u^{\theta_2}(x,t)
&\mbox{for all} \ x \in D^+_t, 
\label{eq:Mono_D+} \\
&& u^{\theta_1}(x,t)&\geq u^{\theta_2}(x,t)
&\mbox{for all} \ x \in D^-_t. 
\label{eq:Mono_D-}
\end{align}
\end{pro}

\begin{proof}
In $D^+$ we observe 
\begin{align*} 
u^{\theta_1}_t
&=H_1(x,\nabla u^{\theta_1})
 +\theta_1 \beta (u^{\theta_1}) h(\nabla u^{\theta_1}) \\
&\leq H_1(x,\nabla u^{\theta_1})
 +\theta_2 \beta (u^{\theta_1}) h(\nabla u^{\theta_1})
\end{align*}
since $h$ is nonnegative.
This implies that $u^{\theta_1}$ is a subsolution of 
\eqref{eq6} with $\theta=\theta_2$ in $D^+$.
Applying the comparison principle to 
a subsolution $(u^{\theta_1})_+$ and a supersolution $(u^{\theta_2})_+$ of \eqref{eq6} with $\theta=\theta_2$, 
we conclude $u^{\theta_1} \leq u^{\theta_2}$ in $D^+$.
By the same argument we see that 
$u^{\theta_2} \leq u^{\theta_1}$ in $D^-$.
\end{proof}

We show that solutions of \eqref{eq:initeq} with small Lipschitz constants
give rise to lower barrier functions in $D^+$ and 
upper barrier functions in $D^-$.
\begin{pro}\label{prop:Bar_coer}
Assume that $h(p)\geq 0$ if $|p|\leq 1$. Let
$w$ be the solution of \eqref{eq:init}.
Then the viscosity solution $u^{\theta}$ of 
\eqref{eq6} and \eqref{eq:initin} satisfies 
\begin{align}
&& u^{\theta}(x,t) &\geq \varepsilon w(x,t)
&\mbox{for all} \ x \in D^+_t, 
\label{eq:Est_w_D+} \\
&& u^{\theta}(x,t) &\leq \varepsilon w(x,t)
&\mbox{for all} \ x \in D^-_t,
\label{eq:Est_w_D-}
\end{align}
where $\varepsilon:=\min\{ 1/\mathrm{Lip}_x[w], \ 1 \}$.
\end{pro}
\begin{proof}
Set $\tilde{w}:=\varepsilon w$.
Since $|\nabla \tilde{w}|=\varepsilon |\nabla w|\leq 1$,
by the assumption of $h$ we observe 
\[ \tilde{w}_t
=H_1(x,t,\nabla \tilde{w}) 
\leq H_1(x,t,\nabla \tilde{w})+\beta (\tilde{w})h(\nabla \tilde{w}) \]
if $\tilde{w}>0$.
In other words, $\tilde{w}$ is a subsolution of \eqref{eq6} in $\{ \tilde{w}>0 \}$.
Applying the comparison principle to a subsolution $(\tilde{w})_+$
and a supersolution $(u^{\theta})_+$ of \eqref{eq6},
we obtain \eqref{eq:Est_w_D+}.
The estimate \eqref{eq:Est_w_D-} is shown in a similar way.
\end{proof}
It remains to construct an upper barrier in $D^+$ and a lower barrier in $D^-$.
In both the cases \eqref{eq:hprop2} and \eqref{eq:hprop3},
the solutions $u^{\theta}$ are dominated 
by the signed distance function $d$ with large coefficient.
In the proof of Proposition \ref{prop:Bar_dist} below,
we use the fact that $d$ is a viscosity supersolution of
\begin{equation} 
d_t=H_1(x-d \nabla d,t, \nabla d) \quad \mbox{in} \ \{ d>0 \}.
\label{eq:eqford}
\end{equation}
This assertion is more or less known 
(see, e.g., \cite[Proof of Theorem 2.2, Step 1--3]{Evans-Soner-Souganidis.1992}), 
but we give its proof in Remark \ref{pf_dsup}
for the reader's convenience.
\begin{pro}\label{prop:Bar_dist}
Assume that $h(p)\leq 0$ if $|p|\geq 1$.
Then the viscosity solution $u^{\theta}$ of 
\eqref{eq6} and \eqref{eq:initin} satisfies 
\begin{align}
&& u^{\theta}(x,t) &\leq  l e^{\LH t} d(x,t)
&\mbox{for all} \ x \in D^+_t, 
\label{eq:Est_plus_D+} \\
&& u^{\theta}(x,t) &\geq  l e^{\LH t} d(x,t)
&\mbox{for all} \ x \in D^-_t,
\label{eq:Est_plus_D-}
\end{align}
where $l:=\max \{ L_0 , \ 1 \}$ and
$\LH$ is the constant in \eqref{itm6}.
\end{pro}
\begin{proof}
Define $\tilde{d}(x,t):= l e^{\LH t} d(x,t)$.
If $d$ is smooth, then
\begin{align*}
\tilde{d}_t - H_1(x,t,\nabla \tilde{d})
&= l L_1 e^{\LH t} d+ l e^{\LH t} d_t - H_1(x,t, l e^{\LH t} \nabla d) \\
&= l e^{\LH t} \{ \LH d + d_t - H_1(x,t,\nabla d) \}.
\end{align*}
We next apply the fact that $d$ is a supersolution of \eqref{eq:eqford} 
to estimate
\begin{align*}
\tilde{d}_t - H_1(x,t,\nabla \tilde{d}) 
&\geq  l e^{\LH t} \{ \LH d + H_1 (x-d \nabla d,t, \nabla d) - H_1(x,t,\nabla d) \} \\
&\geq l e^{\LH t} \{ \LH d -\LH |d\nabla d||\nabla d| \}
\end{align*}
if $d>0$.
Noting that $|\nabla d|=1$, we have
\[ \tilde{d}_t - H_1(x,t,\nabla \tilde{d}) 
\geq  l e^{\LH t} \{ \LH d -\LH d  \}=0. \]
Since $|\nabla \tilde{d}|= l e^{\LH t} |\nabla d|\geq 1$,
we now have $h(\nabla \tilde{d}) \leq 0$ by assumption.
This implies that $\tilde{d}$ is a supersolution 
of \eqref{eq6} in $\{ d>0 \}$.
Even if $d$ is not smooth, 
the same arguments above work in the viscosity sense.\par 
Finally, since $u_0 \le l d_+(\cdot,0)$ in $\mathbf{R}^n$,
applying the comparison principle to 
a subsolution $u^{\theta}$ and a supersolution $(\tilde{d})_+$ of \eqref{eq6},
we conclude \eqref{eq:Est_plus_D+}.
The proof of \eqref{eq:Est_plus_D-} is similar.
\end{proof}
\begin{rem}\label{pf_dsup}
We shall explain why $d$ is a supersolution of \eqref{eq:eqford}. We first note that $d$ is lower semicontinuous in $D^+$ (Theorem \ref{thm:cont_d} \eqref{itm:continuity_1}).
Let $(x_0,t_0) \in \mathbf{R}^n \times (0,T)$ be a point satisfying $d(x_0,t_0)>0$
and take any $(p,\tau) \in \mathcal{J}^- d(x_0,t_0)$.
We choose a smooth function $\phi \in C^1$ such that
$(p,\tau)=(\nabla \phi, \phi_t)(x_0,t_0)$ and
\[
\displaystyle\min_{\mathbf{R}^n \times (0,T)}(d-\phi)=(d-\phi)(x_0,t_0)=0.
\]
Set $d_0:=d(x_0,t_0)$. Since $p\in \mathcal{J}^-(d|_{t=t_0})(x_0)$,
it follows that the closest point of $\Gamma_{t_0}$ to $x_0$ is unique
and that this point is given by $y_0:=x_0-d_0 p \in \Gamma_{t_0}$;
for the proof, refer to 
\cite[Proposition II.2.14]{Bardi-Dolcetta} or 
\cite[Corollary 3.4.5 (i), (ii)]{Cannarsa-Sinestrari}.
We also remark that $|p|=1$.

Define $\psi(x,t):=\phi(x+d_0 p,t)-d_0$.
We now assert
\begin{equation}
\min_{\mathbf{R}^n \times (0,T)}(d_+ -\psi)
=(d_+ -\psi)(y_0,t_0).
\label{eq:min_psi}
\end{equation}
Since $(d_+ -\psi)(y_0,t_0)=0$ and $d_+ \geq 0$,
we only need to show $\{ \psi >0 \} \subset \{ d>0 \}$.
Take a point $(x,t) \in \mathbf{R}^n \times (0,T)$ such that $\psi(x,t)>0$.
We then have $d(x+d_0 p, t) \geq \phi(x+d_0 p,t)>d_0$.
Using the Lipschitz continuity of $d$, we compute
\[ d(x,t)\geq d(x+d_0 p, t)-d_0 |p|>d_0-d_0=0. \]
Thus \eqref{eq:min_psi} is proved.
Let $g(x,t)=\chi_{D^+_t}(x)$.
Then the relation \eqref{eq:min_psi} implies that 
$(p,\tau) \in \mathcal{J}^- g(y_0,t_0)$,
where we applied
$(\nabla \psi, \psi_t)(y_0,t_0)=(\nabla \phi, \phi_t)(x_0,t_0)=(p,\tau)$.
Since the characteristic function $g$ is a supersolution of
\eqref{eq:initeq} (see Remark \ref{rem:inv_char}),
we have
\[ \tau \geq H_1(y_0,t_0,p)=H_1(x_0-d_0 p,t_0,p), \]
which is the inequality we need in order
to conclude that $d$ is a supersolution of \eqref{eq:eqford}.
\end{rem}
\begin{rem}\label{alt_proof}
Another way of proving Proposition \ref{prop:Bar_dist} is using the Lipschitz continuity of solutions of \eqref{eq6} and \eqref{eq:initin} from Proposition \ref{prop2}. Using assumption \eqref{itm6} instead of \eqref{itm6'} in Proposition \ref{prop2}, the Lipschitz estimate for $u^\theta$ reads as follows:
\begin{equation}\label{eq:Lip_est}
|u^\theta (x,t)-u^\theta (y,t)|\leq le^{L_1t}|x-y|\quad \text{ for all } x,y\in \mathbf{R}^n,\, t\in [0,T),
\end{equation}
where $l=\max \{ L_0,1\} $ ($L_0=\mathrm{Lip}[u_0]$). If we take the infimum for all $y\in \Gamma _t$ in \eqref{eq:Lip_est} we get
\[
-le^{L_1t}\mathrm{dist}(x,\Gamma _t)\leq u^\theta (x,t)\leq le^{L_1t}\mathrm{dist}(x,\Gamma _t)\quad \text{ for all } (x,t)\in \mathbf{R}^n\times [0,T),
\]
which implies the relations \eqref{eq:Est_plus_D+} and \eqref{eq:Est_plus_D-}.
\end{rem}
\subsection{Comparison principle for eikonal equations}\label{section_compar_eikonal}

We investigate uniqueness of solutions of the eikonal equation
$|\nabla u|=1$ in a possibly unbounded set.
To establish a convergence to the signed distance function,
we show in the next section
that the limit of the solutions $u^{\theta}$ solves the eikonal equation.
Since the distance function is a solution of the eikonal equation,
the uniqueness result presented below guarantees that 
the limit is the distance function.

We consider the eikonal equation
\begin{equation}
|\nabla u|=1 \quad \mbox{in} \ \Omega
\label{eq:eik_Du1}
\end{equation}
with the boundary condition
\begin{equation}
u=0 \quad \mbox{on} \ \partial \Omega.
\label{eq:eik_zD}
\end{equation}
Here $\Omega \subset \mathbf{R}^n$ is a possibly unbounded open set.
We denote by $d_{\Omega}$ the distance function to $\partial \Omega$,
i.e., $d_{\Omega}(x):=\mathrm{dist}(x,\partial \Omega)$.
It is well known that $d_{\Omega}$ is a viscosity solution 
of \eqref{eq:eik_Du1}; see, e.g., 
\cite[Corollary II.2.16]{Bardi-Dolcetta} or 
\cite[Corollary 3.4.5 (i), (ii) or Remark 5.6.1]{Cannarsa-Sinestrari}.
In other words, the problem \eqref{eq:eik_Du1} with \eqref{eq:eik_zD}
admits at least one viscosity solution.
Comparison principle (and hence uniqueness) of viscosity solutions 
of \eqref{eq:eik_Du1} and \eqref{eq:eik_zD}
is established in \cite{Ishii.1987} when $\Omega$ is bounded.
If $\Omega$ is not bounded,
the uniqueness of solutions does not hold in general;
for instance, when $\Omega=(0,\infty) \subset \mathbf{R}$, 
all of the following functions are solutions:
\[ d_{\Omega}(x)=x, \quad -d_{\Omega}(x)=-x, \quad 
u_a(x)=\min\{ x, \ a-x \} \ (a>0). \]
However, even if $\Omega$ is not bounded,
it turns out that nonnegative solutions 
of \eqref{eq:eik_Du1} and \eqref{eq:eik_zD}
are unique and equal to $d_{\Omega}$.

\begin{lem}\label{lem:CP_eik}
Let $u:\Omega \to \mathbf{R}$.
\begin{enumerate}[label=(\arabic*) ,ref=\arabic*]
	\item\label{item:1side_dist} If $u$ is a viscosity subsolution of \eqref{eq:eik_Du1} and $u^*\leq 0$ on $\partial \Omega$,
then $u^* \leq d_{\Omega}$ in $\Omega$.
	\item\label{item:2side_dist}
If $u$ is a viscosity supersolution of \eqref{eq:eik_Du1} and $u \geq 0$ in $\Omega$,
then $d_{\Omega} \leq u_*$ in $\Omega$.
 \end{enumerate}
\end{lem}
\begin{proof}
(1)
It is known that every subsolution of \eqref{eq:eik_Du1} is 
Lipschitz continuous with Lipschitz constant less than or equal to one,
that is, $|u^*(x)-u^*(y)|\leq |x-y|$ for all $x,y \in \Omega$.
(For the proof see, e.g., \cite[Lemma 5.6]{Giga-Liu-Mitake.2014} or 
\cite[Proof of Proposition 2.1, Step 1]{Monneau-Roquejoffre-RoussierMicho.2013}.)
This yields the inequality $u^* \leq d_{\Omega}$.

(2)
We consider a bounded set $\Omega_R:=\Omega \cap B_R(0)$ with $R>0$.
Define $d_R(x):=\mathrm{dist}(x,\partial \Omega_R)$.
We first note that $u_*\geq 0$ on $\overline{\Omega}$ 
since $u\geq 0$ in $\Omega$, and that
$u_* \geq 0=d_R$ on $\partial \Omega_R$.
Thus, by the comparison principle in bounded sets, we see
$d_R \leq u_*$ in $\Omega_R$.
Finally, sending $R \to \infty$, we conclude $d_\Omega \leq u_*$ in $\Omega$.
\end{proof}
\section{Convergence results}\label{section:converg_results}
Throughout this section we assume \eqref{eq:theor1_assumpt}.
We will first prove Theorem \ref{Theor1} \eqref{itm:main2}, it will then be easier for the reader to understand the proof of Theorem \ref{Theor1} \eqref{itm:main1}, \eqref{itm:main1.5}.\par 
 We introduce a notion of the half-relaxed limits (\cite[Section 6]{Crandall-Ishii-Lions.Ug4}),
which are weak limits of a sequence of functions and
will be used in the proof of the convergence to the distance function.
We define an {\em upper half-relaxed limit} 
$\overline{u}=\limsup^*_{\theta \to \infty} u^{\theta}$ 
and a {\em lower half-relaxed limit}
$\underline{u}=\liminf_{* \theta \to \infty} u^{\theta}$ as
\begin{align*}
\overline{u}(x,t)
:=&\limsup_{(y,s,\theta)\to (x,t,\infty)} u^{\theta}(y,s) \\
 =&\lim_{\delta \to 0} \sup 
\{ u^{\theta}(y,s) \ | \ |x-y|<\delta, \ |t-s|<\delta, \ \theta >1/\delta \}, \\
\underline{u}(x,t)
:=&\liminf_{(y,s,\theta)\to (x,t,\infty)} u^{\theta}(y,s) \\
 =&\lim_{\delta \to 0} \inf 
\{ u^{\theta}(y,s) \ | \ |x-y|<\delta, \ |t-s|<\delta, \ \theta >1/\delta \}.
\end{align*}
Thanks to the existence of barrier functions shown in Section \ref{Sect:Barrier},
we see that, in both the cases \eqref{eq:hprop2} and \eqref{eq:hprop3},
$-\infty< \overline{u}<\infty$ and $-\infty< \underline{u} <\infty$.

The following proposition is true in the general case where the distance function is not necessarily continuous.
\begin{pro}[The zero level set of the relaxed limits]\label{prop:ZLofRLofu}
Assume either \eqref{eq:hprop2} or \eqref{eq:hprop3}. Then
	 \begin{equation}
\{ \underline{u}>0 \}=D^+, \quad 
\{ \underline{u}=0 \} \subset \Gamma, \quad 
\{ \underline{u}<0 \} \supset D^- \label{eq:ZLofLRL}
\end{equation}
and 
\begin{equation} 
\{ \overline{u}>0 \} \supset D^+, \quad 
\{ \overline{u}=0 \} \subset \Gamma, \quad 
\{ \overline{u}<0 \}=D^-. \label{eq:ZLofURL}
\end{equation}
\end{pro}
\begin{proof}
We only show \eqref{eq:ZLofLRL}
since a proof of \eqref{eq:ZLofURL} is similar.
Let $v:=\liminf_{* \theta \to \infty} (u^{\theta})_+$.
Then it is easily seen that $v=(\underline{u})_+$.
From the estimates \eqref{eq:Est_w_D+} and \eqref{eq:Est_plus_D+} of $u^{\theta}$ by barrier functions we derive
\[ 
\varepsilon w_+ \leq (u^{\theta})_+ \leq L d_+
 \]
for some $\varepsilon ,L >0$. Taking the lower half-relaxed limit, we obtain
\[
 \varepsilon w_+ \leq v \leq L (d_+)_* \leq L d_+.
\]
Since $\{ w_+>0 \}=\{ d_+ >0 \}=D^+$,
the above inequalities imply $\{ v>0 \}=D^+$, 
and hence $\{ \underline{u}>0 \}=D^+$.
We similarly have
\[ -L d_- \leq -(u^{\theta})_- \leq -\varepsilon w_-. \]
In this case, however, taking the lower half-relaxed limit yields only 
$\{ \underline{u}<0 \} \supset \{ w<0 \} =D^-$
because $-d_-$ is upper semicontinuous.
The inclusion $\{ \underline{u}=0 \} \subset \Gamma$ is now clear.
\end{proof}

\subsection{Convergence results for continuous distance function}
The following general properties of the relaxed limits 
will be used to prove the convergence of $u^{\theta}$:
\begin{itemize}
\item
Assume that each $u^{\theta}$ is a subsolution (resp. supersolution) 
of the equation $F_{\theta}=0$.
If $F_{\theta}$ converges to some $F$ locally uniformly 
and $\overline{u}< \infty$ (resp. $\underline{u}>-\infty$),
then $\overline{u}$ is a subsolution 
(resp. $\underline{u}$ is a supersolution) of $F=0$.

\item
If $\overline{u}=\underline{u} =:u$ and $-\infty <u<\infty$, 
then $u^{\theta}$ converges to $u$ locally uniformly as $\theta \to \infty$.

\end{itemize}
See \cite[Lemma 6.1, Remark 6.4]{Crandall-Ishii-Lions.Ug4} for the proofs.\par
Assume that $d$ is continuous in $\mathbf{R}^n \times (0,T)$, in particular, we can now use the additional upper-semicontinuity property of $d_+$ and $d_-$. Proceeding in a similar way as in Proposition \ref{prop:ZLofRLofu} we can show
  \begin{equation}\label{eq:similar_level_sets}
  \Gamma=\{\overline{u}=0\}=\{\underline{u}=0\}, \quad
D^{\pm}=\{\pm \overline{u}>0\}=\{\pm\underline{u}>0\} . 
  \end{equation}

%
%
%
%

\begin{proof}[Proof of Theorem \ref{Theor1} \eqref{itm:main2}]
1.
For $\theta >0$ we define 
\[ F_{\theta}(x,t,r,p,\tau)
:=\frac{1}{\theta}\{ \tau-H_1(x,t,p) \}-\beta (r)h(p). \]
Then $u^{\theta}$ is a viscosity solution of the equation
$F_{\theta}(x,t,u,\nabla u,u_t)=0$
in $\mathbf{R}^n \times (0,T)$.
Since $F_{\theta}$ converges to $-\beta (r)h(p)$
locally uniformly as $\theta \to \infty$,
it follows that $\overline{u}$ and $\underline{u}$ are,
respectively, a viscosity subsolution and a viscosity supersolution of 
$-\beta (u)h(\nabla u)=0$ in $\mathbf{R}^n \times (0,T)$.

Recall that $h$ satisfies \eqref{eq:hprop2}.
Since $\beta (\overline{u})>0$ in $D^+$
and $\beta (\overline{u})<0$ in $D^-$ by \eqref{eq:similar_level_sets},
we see that $\overline{u}$ is a subsolution of 
\begin{equation} 
|\nabla u(x,t)|=1 \quad \mbox{in} \ D^+
\label{eq:Eik_UD+}
\end{equation}
and 
\begin{equation} 
-|\nabla u(x,t)|=-1 \quad \mbox{in} \ D^-
\label{eq:Eik_UD-}
\end{equation}
as a function of $(x,t)$.
(Note that these two equations are different in the viscosity sense.)
Similarly, $\underline{u}$ is a supersolution of
both \eqref{eq:Eik_UD+} and \eqref{eq:Eik_UD-}.
Thus, for each fixed $t_0\in (0,T)$, 
$\overline{u}|_{t=t_0}$ and $\underline{u}|_{t=t_0}$ are,
respectively, a subsolution and a supersolution of 
\begin{equation} 
|\nabla u(x)|=1 \quad \mbox{in} \ D^+_{t_0}
\label{eq:Eik_D+}
\end{equation}
as a function of $x$.
(See Remark \ref{rem:xt_x} for the details.)
By Lemma \ref{lem:CP_eik} we obtain
\[
\overline{u}|_{t=t_0} \leq d(\cdot ,t_0) \leq \underline{u}|_{t=t_0}\quad \text{ in } D^+_{t_0}
\]
and hence 
\[
d=\overline{u}=\underline{u}\quad \text{ in } D^+.
\]
This implies that $u^{\theta} \to d$ locally uniformly in $D^+$.
For $D^-$ we notice that if $\overline{u}(\cdot ,t_0)$ is a subsolution of $-|\nabla u|= -1$ then $-\overline{u}(\cdot ,t_0)$ is a supersolution of $|\nabla u|= 1$, hence a comparison with $-d$ this time gives the desired result.
\end{proof}
\begin{rem}\label{rem:xt_x}
We claim that, 
if $u=u(x,t)$ is a subsolution of \eqref{eq:Eik_UD+},
then $u|_{t=t_0}$ is a subsolution of \eqref{eq:Eik_D+}
for a fixed $t_0\in (0,T)$.
To show this, we take a test function $\phi \in C^1(\mathbf{R}^n)$
such that $\max_{\mathbf{R}^n}(u|_{t=t_0}-\phi)=u(x_0,t_0)-\phi(x_0)$ for $x_0 \in D^+_{t_0}$.
We may assume that this is a strict maximum.
Next define $\psi_M(x,t):=\phi(x)+M(t-t_0)^2$.
We then have
\[ \left( \liminf_{M \to \infty} \hspace{-0.7mm}_* 
\psi_M \right)(x,t)=
\begin{cases} \phi(x) & \mbox{if} \ t=t_0, \\ 
\infty & \mbox{if} \ t \neq t_0, \end{cases} \]
so that $u-(\liminf_* \psi_M)$ has a strict maximum 
over $\mathbf{R}^n \times (0,T)$ at $(x_0,t_0)$.
By \cite[Lemma 2.2.5]{Giga.See}
there exist sequences $\{ M_n \}_{n=1}^{\infty} \subset (0,\infty)$
and $\{ (x_n,t_n) \}_{n=1}^{\infty} \subset \mathbf{R}^n \times (0,T)$
such that 
$M_n \to \infty$, $(x_n,t_n) \to (x_0,t_0)$ as $n \to \infty$ and
$u-\psi_{M_n}$ has a local maximum at $(x_n,t_n)$.
Since $u$ is a subsolution of \eqref{eq:Eik_UD+}, we have
\[ 1 \geq |\nabla \psi_{M_n}(x_n,t_n) |=|\nabla \phi(x_n)|. \]
Sending $n \to \infty$ implies $|\nabla \phi(x_0)| \leq 1$;
namely, $u|_{t=t_0}$ is a subsolution of \eqref{eq:Eik_D+}.
\end{rem}

\subsection{Convergence results for general distance functions}
 For the case where the distance $d$ is not necessarily continuous we can only compare the half-relaxed limits with the distance function in certain domains due to the fact that only the inclusions in Proposition \ref{prop:ZLofRLofu} are true.
\begin{lem}[Comparison with the distance]\label{Lemma:comparison_discont}
Assume that either \eqref{eq:hprop2} or \eqref{eq:hprop3} hold. Then
\begin{enumerate}[label=(\arabic*) ,ref=\arabic*]
	\item\label{item2} \begin{align}
d \leq \underline{u} \quad \mbox{in $D^+$}, \label{eq:duubar_gen} \\
\overline{u} \leq d \quad \mbox{in $D^-$}. \label{eq:uobard_gen}
\end{align}
	\item\label{item3} For every $t \in (0,T)$,
\begin{align}
\underline{u}(\cdot ,t)=0 \quad \mbox{on $\partial D^+_t$}, \label{eq:uubar_z} \\
\overline{u}(\cdot ,t)=0  \quad \mbox{on $\partial D^-_t$}. \label{eq:uobar_z}
\end{align}
 \end{enumerate}
\end{lem}
\begin{proof}
We give proofs of \eqref{eq:duubar_gen} and \eqref{eq:uubar_z} since \eqref{eq:uobard_gen} and \eqref{eq:uobar_z} can be shown in similar ways.\par 
\textbf{\eqref{item2}}
In the same manner as in the proof of Theorem \ref{Theor1} \eqref{itm:main2},
it follows that $\underline{u}(\cdot,t)$ is a viscosity supersolution
of \eqref{eq:Eik_D+} in $D^+_t$.
Since $\underline{u}(\cdot,t)>0$ in $D^+_t$ by \eqref{eq:ZLofLRL},
the comparison principle (Lemma \ref{lem:CP_eik} \eqref{item:2side_dist}) implies that 
$d(\cdot,t)\leq \underline{u}(\cdot,t)$ in $D^+_t$.\par 
\textbf{\eqref{item3}}
By \eqref{eq:Est_plus_D+} and \eqref{eq:Est_plus_D-} we have
\begin{equation}\label{eq:Lem_compar_disc_1}
-L d_- \leq u^{\theta} \leq L d_+,
\end{equation}
where $L>0$ is a constant.
Taking the lower half-relaxed limit at $(x,t)$, we obtain
\begin{equation} 
-L (d_-)^*(x,t) \leq \underline{u}(x,t) 
\leq L (d_+)_*(x,t)\leq Ld_+(x,t).  
\label{eq:dm_uubar_dp}
\end{equation}
Let $x \in \partial D^+_t$.
Then the right-hand side of \eqref{eq:dm_uubar_dp} is 0 since $x\in \Gamma _t$.  We next study the limit of 
$d_-(y,s)=\mathrm{dist}(y, D^+_s \cup \Gamma_s)$ on the left-hand side.
Since $x \in \partial D^+_t \subset \overline{D^+_t} \subset \overline{\mathrm{int}(D^+_t \cup \Gamma_t)}$,
it is not an extinction point (Definition \ref{defn:Expt}) by Proposition \ref{prop:suf_nonE}.
Therefore Theorem \ref{thm:cont_d} \eqref{itm:continuity_3} ensures that $d_-$ is continuous at $(x,t)$.
This implies that the left-hand side of \eqref{eq:dm_uubar_dp} is 0,
and hence the conclusion follows.
\end{proof}
\begin{proof}[Proof of Theorem \ref{Theo1.1}]
Since \eqref{eq:hprop3} holds,
the monotonicity of $u^{\theta}$ (Proposition \ref{prop:mono})
yields the following representations:
\[ \underline{u}(x,t)= \sup_{\theta >0} u^{\theta}(x,t) 
\quad \mbox{for} \ x \in D^+_t, \quad 
\overline{u}(x,t)= \inf_{\theta >0} u^{\theta}(x,t)
\quad \mbox{for} \ x \in D^-_t. \]
These relations and Lemma \ref{Lemma:comparison_discont} \eqref{item2} conclude the proof.
\end{proof}
For the equation with $h$ satisfying \eqref{eq:hprop3},
Theorem \ref{Theo1.1} guarantees only the one side inequality between  
the supremum of $u^{\theta}$ and the signed distance function $d$.
However, as the next example shows,
the opposite inequality is not true in general 
even if the initial datum is smaller than the distance function.
\begin{ex}\label{exmp:no_convergence}
Let us consider \eqref{eq:init} 
for the Hamiltonian of the form \eqref{eq:exmpH_1} with $c(x)=(1-|x|)_+ +1$.
We take the initial datum $u_0$ as $u_0(x)=(1-|x|)_+$.
The unique viscosity solution $w$ of this initial value problem 
is given as the value function \eqref{eq:rep_opt}.
In this case the optimal control is the one 
that leads to a straight trajectory with the maximal speed
before it comes to the origin and stays there after that moment.
Thus direct calculations yield the following simplified representation of $w$:
\[ w(x,t)=\begin{cases} 
1 & \mbox{if} \ |x|\leq 2(1-e^{-t}), \\
(2-|x|)e^t -1 & \mbox{if} \ 2(1-e^{-t}) \leq |x| \leq 1, \\
e^{t-|x|+1}-1 & \mbox{if} \ 1 \leq |x| \leq t+1, \\
0 & \mbox{if} \ t+1\leq |x|
\end{cases} \quad \mbox{for} \ t \leq \log 2, \]
and
\[ w(x,t)=\begin{cases} 
1 & \mbox{if} \ |x|\leq t+1-\log 2, \\
e^{t-|x|+1}-1 & \mbox{if} \ t+1-\log 2\leq |x| \leq t+1, \\
0 & \mbox{if} \ t+1\leq |x|
\end{cases} \quad \mbox{for} \ t \geq \log 2. \]
See Figure \ref{fig:d4_solw}. In particular, we have $w(x,t)=1$ if $|x|=t+1-\log 2 \geq 1$.
Also, $\{ w=0 \}=\{ |x|\geq t+1 \}$ and 
the signed distance function $d$ to the interface is
$d(x,t)=(t+1-|x|)_+$.
We thus have $d(x,t)=\log 2$ if $|x|=t+1-\log 2$,
and so
\begin{equation} 
d(x,t)=\log 2 <1=w(x,t) 
\quad \mbox{if} \ |x|=t+1-\log 2 \geq 1. 
\label{eq:dlogw}
\end{equation}
Since the solution $w$ is non-negative, 
it is a viscosity subsolution of \eqref{eq6} 
with $h \geq 0$ for every $\theta>0$.
Accordingly, $w \leq u^{\theta}$ by the comparison principle.
From \eqref{eq:dlogw} it follows that
\[ d(x,t)=\log 2 <1\leq u^{\theta}(x,t) 
\quad \mbox{if} \ |x|=t+1-\log 2 \geq 1, \]
which implies that the inequality $d \geq \sup_{\theta>0}u^{\theta}$
does not hold on the whole space.

We also remark that, for $\gamma \in (\log 2,1)$,
the inequality $d(x,t)<\gamma w(x,t)$ holds 
if $|x|=t+1-\log 2 \geq 1$ and that
$\gamma w$ is a solution of \eqref{eq:initeq} with the initial datum $\gamma u_0$.
From this we see that $u^{\theta}$ can be greater than $d$ at some point
even if we take an initial datum which is strictly less than $d(x,0)$
in $\{ d(\cdot,0)>0 \}$.

\begin{figure}[htbp]
\begin{center}
\includegraphics{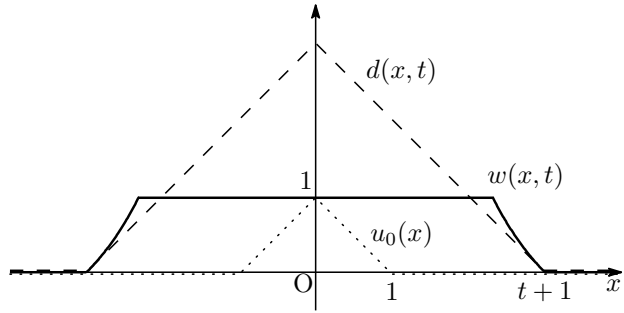}
\end{center}
\caption{The graph of $w$ when $t \ge \log 2$.}
\label{fig:d4_solw}
\end{figure}

\end{ex}
In the rest of this subsection we will assume that $h$ satisfies the assumption \eqref{eq:hprop2}.
We now introduce several notions of half-relaxed limits.
Let $(x,t) \in \mathbf{R}^n \times (0,T)$.
We define an {\it upper} and a {\it lower half-relaxed limit from below in time} by, respectively,
\[ \overline{u}'(x,t)
:=\limsup_{\begin{subarray}{c}(y,s,\theta) \to (x,t,\infty) \\ s\leq t \end{subarray}} u^{\theta}(y,s),
\quad \underline{u}'(x,t)
:=\liminf_{\begin{subarray}{c}(y,s,\theta) \to (x,t,\infty) \\ s\leq t \end{subarray}} u^{\theta}(y,s). \]
An {\it upper} and a {\it lower half-relaxed limit at a fixed time}
are, respectively, given as
\[ \overline{u|_t}(x)
:=\limsup_{(y,\theta) \to (x,\infty)} u^{\theta}(y,t),
\quad \underline{u|_t}(x)
:=\liminf_{(y,\theta) \to (x,\infty)} u^{\theta}(y,t). \]
By definitions we have
\begin{equation} 
\underline{u}(x,t) \leq \underline{u}'(x,t) \leq \underline{u|_t}(x) 
\leq \overline{u|_t}(x) \leq \overline{u}'(x,t) \leq \overline{u}(x,t)
\label{eq:limits_general}
\end{equation}
for all $(x,t) \in \mathbf{R}^n \times (0,T)$.\par 
The next proposition is a crucial step
in proving a convergence to the signed distance function in a weak sense.
\begin{pro}\label{prop:sub_fixed_t}
Assume either \eqref{eq:hprop2} or \eqref{eq:hprop3}. Then the functions
$\overline{u}'(\cdot,t)$ and $\underline{u}'(\cdot,t)$ are,
respectively, a viscosity subsolution of \eqref{eq:Eik_UD+} in $D^+_t$
and a viscosity supersolution of \eqref{eq:Eik_UD-} in $D^-_t$
for every $t \in (0,T)$.
\end{pro}

 
\begin{proof}
Fix $\hat{t} \in (0,T)$ and let us prove that
$\overline{u}'(\cdot,\hat{t})$ is a viscosity subsolution of \eqref{eq:Eik_UD+} in $D^+_{\hat{t}}$.

1.
We first introduce an upper half-relaxed limit of $u^{\theta}$ 
in $\mathbf{R}^n \times (0,\hat{t}]$.
For $(x,t) \in \mathbf{R}^n \times (0,\hat{t}]$ we define
\[ \overline{v}(x,t)
:=\limsup_{\begin{subarray}{c}(y,s,\theta) \to (x,t,\infty) \\ s\leq \hat{t} \end{subarray}} u^{\theta}(y,s), \]
which is an upper semicontinuous function on $\mathbf{R}^n \times (0,\hat{t}]$.
By definition we have
\[ \overline{v}(x,t)=\begin{cases} 
\overline{u}(x,t) & \mbox{if} \ t<\hat{t}, \\
\overline{u}'(x,t) & \mbox{if} \ t=\hat{t}.
\end{cases} \]

2.
Take $z \in D^+_{\hat{t}}$ and $\psi \in C^1(\mathbf{R}^n)$ such that 
$\overline{u}'(\cdot,\hat{t})-\psi$ attains a maximum at $z$ over $\mathbf{R}^n$.
As usual we may assume that this is a strict maximum, 
and note that,
by \eqref{eq:duubar_gen} and \eqref{eq:limits_general},
\begin{equation} 
0<d(z,\hat{t})\leq \underline{u}(z,\hat{t})\leq \overline{u}'(z,\hat{t}). 
\label{eq:baru_plus}
\end{equation}
We now define $\phi^{\theta}(x,t):=\psi(x)-\sqrt{\theta}(t-\hat{t})$ and
\[ \phi(x,t):=\begin{cases} 
+\infty & \mbox{if} \ t<\hat{t}, \\
\psi(x) & \mbox{if} \ t=\hat{t}.
\end{cases} \]
Then $\overline{v}-\phi$ attains its strict maximum at $(z,\hat{t})$ over $\mathbf{R}^n \times (0,\hat{t}]$,
and $u^{\theta}-\phi^{\theta} \to \overline{v}-\phi$ 
in the sense of the upper half-relaxed limit on $\mathbf{R}^n \times (0,\hat{t}]$.
Thus, by \cite[Lemma 2.2.5]{Giga.See} there exist sequences 
$\{ \theta_j \}_{j=1}^{\infty} \subset (0,\infty)$ and 
$\{ (x_j,t_j) \}_{j=1}^{\infty} \subset \mathbf{R}^n \times (0,\hat{t}]$
such that $\theta_j \to \infty$, $(x_j,t_j)\to (z,\hat{t})$ and 
$(u^{\theta_j}-\phi^{\theta_j})(x_j,t_j) \to (\overline{v}-\phi)(z,\hat{t})$ as $j \to \infty$.

We now claim 
\begin{equation} 
\overline{u}'(z,\hat{t})
=\lim_{j \to \infty} u^{\theta_j}(x_j,t_j). \label{eq:lim_u_j}
\end{equation}
Observe
\begin{align*}
u^{\theta_j}(x_j,t_j)
&=\{ (u^{\theta_j}-\phi^{\theta_j})(x_j,t_j)-(\overline{v}-\phi)(z,\hat{t}) \}
 +\phi^{\theta_j}(x_j,t_j)+(\overline{v}-\phi)(z,\hat{t}) \\
&=\{ (u^{\theta_j}-\phi^{\theta_j})(x_j,t_j)-(\overline{v}-\phi)(z,\hat{t}) \}
 +\{ \psi(x_j)-\psi(z) \} + \overline{u}'(z,\hat{t})-\sqrt{\theta}(t_j-\hat{t}) \\
&\geq \{ (u^{\theta_j}-\phi^{\theta_j})(x_j,t_j)-(\overline{v}-\phi)(z,\hat{t}) \}
 +\{ \psi(x_j)-\psi(z) \} + \overline{u}'(z,\hat{t}).
\end{align*}
This implies  
$\displaystyle\liminf_{j \to \infty} u^{\theta_j}(x_j,t_j) \geq \overline{u}'(z,\hat{t})$.
The opposite relation
$\displaystyle\limsup_{j \to \infty} u^{\theta_j}(x_j,t_j) \leq \overline{u}'(z,\hat{t})$
follows from the definition of $\overline{u}'$,
and therefore \eqref{eq:lim_u_j} is proved. 

3.
Since $u^{\theta}$ is a viscosity solution of \eqref{eq6}
in $\mathbf{R}^n \times (0,\hat{t})$
and since the viscosity property is extended up to the terminal time $t=\hat{t}$ (\cite[Section 7]{Chen_Giga_Goto}),
we have
\[ \phi^{\theta_j}_t(x_j,t_j) 
\leq H_1(x_j,t_j,\nabla\phi^{\theta_j}(x_j,t_j))
 +\theta \beta(u^{\theta_j}(x_j,t_j))h(\nabla\phi^{\theta_j}(x_j,t_j)). \]
By the definition of $\phi^{\theta}$, this is equivalent to
\[ -\sqrt{\theta} 
\leq H_1(x_j,t_j,\nabla\psi(x_j))
 +\theta \beta(u^{\theta_j}(x_j,t_j))h(\nabla\psi(x_j)). \]
Dividing both the sides by $\theta$ and sending $\theta \to \infty$, we obtain
\[ 0 \leq \beta(\overline{u}'(z,\hat{t}))h(\nabla\psi(z)), \]
where we have used \eqref{eq:lim_u_j}.
Since $\beta (\overline{u}'(z,\hat{t}))>0$ by \eqref{eq:baru_plus},
using the assumption on $h$, we conclude that $|\nabla \psi(z)| \leq 1$.
\end{proof}

As a consequence of Proposition \ref{prop:sub_fixed_t}, we obtain

\begin{theo}\label{thm:mainconv_discont}
Assume \eqref{eq:hprop2}. Then the following hold.
\begin{enumerate}[label=(\arabic*) ,ref=\arabic*]
	\item\label{itm:compar_1}$\underline{u}(\cdot,t)=d(\cdot,t)$ on $\overline{D^+_t}$ and
$\overline{u}(\cdot,t)=d(\cdot,t)$ on $\overline{D^-_t}$ for every $t \in (0,T)$.
	\item\label{itm:compar_2} $\overline{u}'=\underline{u}'=d$ in $\mathbf{R}^n \times (0,T)$, i.e.,
\[ \lim_{\begin{subarray}{c}(y,s,\theta) \to (x,t,\infty) \\ s\leq t \end{subarray}} u^{\theta}(y,s)=d(x,t) 
\quad \mbox{for all $(x,t) \in \mathbf{R}^n \times (0,T)$}. \]
\item\label{itm:compar_3}$\overline{u|_t}=\underline{u|_t}=d(\cdot, t)$ in $\mathbf{R}^n$
for every $t \in (0,T)$, i.e.,
$u^{\theta}(\cdot,t)$ converges to $d(\cdot,t)$
locally uniformly in $\mathbf{R}^n$ for every $t \in (0,T)$.
 \end{enumerate}
\end{theo}
\begin{proof}
1.
We first note that \eqref{eq:Lem_compar_disc_1} yields 
\begin{align}
\overline{u}'=\underline{u}'&=0 \quad \mbox{on} \ \Gamma, 
\label{eq:uvponG} \\
\underline{u|_t}=\overline{u|_t}&=0 \quad \mbox{on} \ \Gamma.
\label{eq:utonG}
\end{align}
Indeed, for $(x,t) \in \Gamma$, 
taking the upper and lower half-relaxed limit from below in time in \eqref{eq:Lem_compar_disc_1},
we see that $\overline{u}'(x,t)=\underline{u}'(x,t)=0$
since $d$ is continuous from below in time by Theorem \ref{thm:cont_d} \eqref{itm:continuity_2}.
Similarly, \eqref{eq:utonG} follows from the continuity of $d(\cdot, t)$.
Thus \eqref{itm:compar_2} and \eqref{itm:compar_3} were proved on $\Gamma$.
The equalities in \eqref{itm:compar_1} on $\partial D^+_t$ or $\partial D^-_t$
are consequences of Lemma \ref{Lemma:comparison_discont} \eqref{item3}.

2.
It remains to prove \eqref{itm:compar_1}--\eqref{itm:compar_3} in $D^+_t$ and $D^-_t$.
Recall that $\overline{u}'(\cdot,t)$ and $\underline{u}'(\cdot,t)$ are,
respectively, a viscosity subsolution of \eqref{eq:Eik_UD+} in $D^+_t$
and a viscosity supersolution of \eqref{eq:Eik_UD-} in $D^-_t$ by Proposition \ref{prop:sub_fixed_t}.
Since \eqref{eq:uvponG} holds,
the comparison result (Lemma \ref{lem:CP_eik} \eqref{item:1side_dist}) implies that
\begin{align}
\overline{u}'(\cdot,t) \le d(\cdot,t) \quad \mbox{in $D^+_t$,}\label{eq:417}\\
 d(\cdot ,t) \le \underline{u}'(\cdot,t) \quad \mbox{in $D^-_t$.} \label{eq:418}
\end{align}
Combining \eqref{eq:duubar_gen}, \eqref{eq:limits_general} and \eqref{eq:417},
we obtain
\[ 0<d(\cdot,t)=\underline{u}(\cdot ,t)=\underline{u}'(\cdot ,t)
=\underline{u|_t}=\overline{u|_t}=\overline{u}'(\cdot,t) 
\quad \mbox{in $D^+_t$.} \]
In the same manner, we see 
\[ 0>d(\cdot,t)=\underline{u}'(\cdot ,t)=\underline{u|_t}
=\overline{u|_t}=\overline{u}'(\cdot,t)=\overline{u}(\cdot ,t) 
\quad \mbox{in $D^-_t$.} \]
The two relations above conclude the proof.
\end{proof}
This concludes the proof of Theorem \ref{Theor1} \eqref{itm:main1} and \eqref{itm:main1.5}.

\section{Continuity of distance functions}\label{Section:continuity_distance}
Throughout this section we study only non-negative distance functions.
Namely, we assume $D^- =\emptyset$ so that 
$d(x,t)=\mathrm{dist}(x,\Gamma_t) \geq 0$ 
for all $(x,t) \in \mathbf{R}^n \times (0,T)$.
Also, we simply write $D_t=D^+_t$ and $D=D^+$.
In the general case where $d$ can take negative values,
we decompose $d$ as $d=d_+-d_-$ and apply the following results to $d_+$ and $d_-$. 
\subsection{Finite Propagation}\label{Section:finite_propagat}
In order to study the continuity of distance functions, we first prepare a property of finite propagation for the Hamilton-Jacobi equation \eqref{eq:init}. For this property, the assumption \eqref{itm6.1}, the Lipschitz continuity of $H_1$ in $p$ plays an important role, though we omit the details in this paper.\par 
Let $(x,t) \in \mathbf{R}^n \times (0,T)$ and $r>0$. 
We define a cone as 
\[ \mathcal{C}_{(x,t)}^r:=
\bigcup_{0<\tau <r} B_{r-\tau }(x) \times \left\{ t+\frac{\tau }{L_2} \right\}. \]

\begin{theo}[Local comparison principle]\label{thm:LocCP}
Let $(x,t) \in \mathbf{R}^n \times (0,T)$, $r>0$ and set $\mathcal{C}:=\mathcal{C}_{(x,t)}^r$.
If $u,v \in C(\overline{\mathcal{C}})$ are, respectively, 
a viscosity sub- and supersolution of \eqref{eq:initeq} in $\mathcal{C}$
and $u(\cdot,t) \leq v(\cdot,t)$ in $\overline{B_r(x)}$,
then $u \leq v$ in $\overline{\mathcal{C}}$.
\end{theo}

See \cite[Theorem III.3.12, (Exercise 3.5)]{Bardi-Dolcetta} or in \cite[Theorem 5.3]{Achdou_Barles_etc} for the proof.
As a consequence of Theorem \ref{thm:LocCP} we obtain

\begin{pro}[Finite propagation]\label{prop:FP_ball}
Let $(x,t) \in \mathbf{R}^n \times (0,T)$ and $r>0$.
\begin{enumerate}[label=(\arabic*) ,ref=\arabic*]
	\item\label{itm:FP_1} If $\overline{B_r(x)} \subset D_t$, 
then $\overline{\mathcal{C}_{(x,t)}^r} \subset D$.
	\item\label{itm:FP_2} If $\overline{B_r(x)} \subset \Gamma_t$, 
then $\overline{\mathcal{C}_{(x,t)}^r} \subset \Gamma$.
 \end{enumerate}
\end{pro}

\begin{proof}
Let $w$ be the solution of \eqref{eq:init}.

\eqref{itm:FP_1}
Set $\alpha:=\displaystyle\min_{\overline{B_r(x)} \times \{ t \}} w>0$,
and define $u(x,t):=\alpha$, which is a constant function satisfying 
$u(\cdot ,t)\leq w(\cdot ,t)$ in $\overline{B_r(x)}$.
Moreover, $u$ is a solution of \eqref{eq:initeq} 
by the geometricity of $H_1$.
Therefore Theorem \ref{thm:LocCP} implies that $u \leq w$ 
in $\overline{\mathcal{C}_{(x,t)}^r}$.
The positivity of $u$ implies the conclusion.

\eqref{itm:FP_2}
The proof is similar to \eqref{itm:FP_1}.
We compare $w$ with $u(x,t):=0$ both from above and from below 
to conclude that $0=u \leq w \leq u=0$ in $\overline{\mathcal{C}_{(x,t)}^r}$.
\end{proof}

\subsection{Continuity properties}
We first introduce a notion of extinction points.

\begin{defi}[Extinction point]\label{defn:Expt}
Let $x \in \Gamma _t$.
We say that $x$ is an {\em extinction point} 
if there exist $\varepsilon,\delta >0$ such that 
$\overline{B_{\varepsilon}(x)} \times (t,t+\delta] \subset D$.
\end{defi}

For example the point $0 \in \Gamma_1$ in Example \ref{exmp:sdisc}
is an extinction point.
We remark that $x \in \Gamma _t$ is non-extinction point if and only if
there exists a sequence $\{ (x_j,t_j) \}_{j=1}^{\infty}$ such that
$(x_j,t_j) \to (x,t)$ as $j \to \infty$,
$x_j \in \Gamma_{t_j}$ and $t_j>t$ for all $j$.
Define $E_t \subset \mathbf{R}^n$ as
the set of all extinction points at time $t \in (0,T)$
and $N_t(x)$ as the set of all the nearest points from $x \in \mathbf{R}^n$ to $\Gamma_t$, i.e.,
\[ N_t(x):=\{ z \in \Gamma _t \ | \ d(x,t)=|x-z| \}. \]
Note that we always have $N_t(x) \neq \emptyset$ by \eqref{eq:sec_2_1}.

\begin{theo}[Continuity properties of the distance function]\label{thm:cont_d}
\begin{enumerate}[label=(\arabic*) ,ref=\arabic*]
	\item\label{itm:continuity_1} $d$ is lower semicontinuous in $\mathbf{R}^n \times (0,T)$.
	\item\label{itm:continuity_2} $d$ is continuous from below in time, i.e.,
\[ d(x,t)=\lim_{\begin{subarray}{c} (y,s) \to (x,t) \\ s\leq t \end{subarray}}d(y,s) 
\quad \mbox{for all} \ (x,t)\in \mathbf{R}^n \times (0,T). \]
\item\label{itm:continuity_3}Let $(x,t) \in \mathbf{R}^n \times (0,T)$.
Then $d$ is continuous at $(x,t)$ if and only if 
$N_t(x) \setminus E_t \neq \emptyset$.
 \end{enumerate}
\end{theo}



\begin{proof}
\eqref{itm:continuity_1}
The proof can be found in \cite[Proposition 2.1]{Evans-Soner-Souganidis.1992}.

\eqref{itm:continuity_2}
1.
Suppose by contradiction that 
$d$ is not continuous at $(x,t)$ from below in time.
Since $d$ is lower semicontinuous by \eqref{itm:continuity_1}, we would have a sequence 
$\{ (x_j,t_j) \}_{j=1}^{\infty}$ such that
$(x_j,t_j) \to (x,t)$ as $j \to \infty$, $t_j <t$ and
\[ \lim_{j\to \infty}d(x_j,t_j)>d(x,t). \]
Set $\alpha:=\{ \lim_{j\to \infty}d(x_j,t_j)-d(x,t) \}/4 >0$.
Without loss of generality we may assume that 
$d(x_j,t_j)-d(x,t)\geq 3\alpha$ and $|x_j-x| \leq \alpha$
for all $j \geq 1$.

2.
Take any $z \in N_t(x)$.
We claim that 
\begin{equation} 
d(y,t_j)\geq \alpha \quad 
\mbox{for all $y\in \overline{B_{\alpha}(z)}$ and $j\geq 1$.} 
\label{eq:dal_claim}
\end{equation}
Since $d(\cdot, t_j)$ is a Lipschitz continuous function
with the Lipschitz constant 1, we calculate
\begin{align*}
d(y,t_j)
&\geq d(x_j,t_j)-|x_j-y| \\
&\geq \{ d(x,t)+3\alpha \}-(|x_j-x|+|x-z|+|z-y|) \\
&\geq \{ d(x,t)+3\alpha \}-(\alpha+|x-z|+\alpha) \\
&\geq \alpha,
\end{align*}
which yields \eqref{eq:dal_claim}.
By \eqref{eq:dal_claim} we have 
$\overline{B_{\alpha}(z)} \times \{ t_j \} \subset D$.
Thus Proposition \ref{prop:FP_ball} \eqref{itm:FP_1} implies that
\begin{equation}
\overline{\mathcal{C}_{(z,t_j)}^{\alpha}} \subset D.
\label{eq:cone_D}
\end{equation}
Since $t_j \uparrow t$ as $j \to \infty$, 
we have $(z,t) \in \overline{\mathcal{C}_{(z,t_j)}^{\alpha}}$ for $j$ large,
and therefore $z \in D_t$ by \eqref{eq:cone_D}.
However, this contradicts the fact that $z \in \Gamma_t$.

\eqref{itm:continuity_3}
1.
We first assume that $d$ is continuous at $(x,t)$.
Take any sequence $\{ (x_j,t_j) \}_{j=1}^{\infty}$ such that
$(x_j,t_j) \to (x,t)$ as $j \to \infty$ and $t_j >t$.
By continuity we have $d(x_j,t_j) \to d(x,t)$ as $j\to \infty$.
We now take $z_j \in N_{t_j}(x_j)$ for each $j$.
Then $\{ z_j \}$ is bounded.
Indeed, since 
$|z_j|\leq |x|+|x-x_j|+|x_j-z_j|$
and $|x-x_j| \to 0$, $|x_j-z_j|=d(x_j,t_j) \to d(x,t)$ as $j \to \infty$,
we see that $\{ z_j \}$ is bounded.
From this $z_j$ subsequently converges to some $\bar{z}$ as $j \to \infty$,
where we use again the index $j$. It is easy to see that $\bar{z}\in \Gamma _t$.

Let us show $\bar{z} \in N_t(x) \setminus E_t$.
Taking the limit in $d(x_j,t_j)=|x_j-z_j|$,
we obtain $d(x,t)=|x-\bar{z}|$,
which implies that $\bar{z} \in N_t(x)$.
Also, 
since $z_j \in \Gamma_{t_j}$ and $t_j \downarrow t$ as $j \to \infty$,
it follows that $\bar{z}$ is not an extinction point,
and hence we conclude that $N_t(x) \setminus E_t \neq \emptyset$.

2.
We next assume that $d$ is not continuous at $(x,t)$.
By \eqref{itm:continuity_1} and \eqref{itm:continuity_2} we have some sequence 
$\{ (x_j,t_j) \}_{j=1}^{\infty}$ such that
$(x_j,t_j) \to (x,t)$ as $j \to \infty$, $t_j >t$ and
\[ \lim_{j\to \infty}d(x_j,t_j)>d(x,t). \]
We now argue in a similar way to the proof of (2),
so that we obtain \eqref{eq:cone_D} for any $z \in N_t(x)$.
Therefore
\[ \bigcup_{j=1}^{\infty} \overline{\mathcal{C}_{(z,t_j)}^{\alpha}} \subset D, \]
and it is easily seen that there exist $\varepsilon, \delta>0$ such that
\[ \overline{B_{\varepsilon}(z)} \times (t,t+\delta] \subset
\bigcup_{j=1}^{\infty} \overline{\mathcal{C}_{(z,t_j)}^{\alpha}}. \]
We thus conclude that $z$ is an extinction point,
and hence $N_t(x) \setminus E_t =\emptyset$.
\end{proof}
\begin{rem}\label{rem:conection_main_theorem}
\begin{enumerate}[label=(\arabic*) ,ref=\arabic*]
	\item\label{itm:rem_conection_1} Theorem \ref{thm:cont_d} \eqref{itm:continuity_3} implies that if every $x\in \Gamma _t$ with $t\in (0,T)$ is a non extinction point then the distance function $d$ is continuous in $\mathbf{R}^n\times (0,T)$ and hence Theorem \ref{Theor1} \eqref{itm:main2} holds.
	  \item\label{itm:rem_connection_3} If $d$ is discontinuous at times $0<t_1<t_2<...<t_m<T$ (at one or more points in $\mathbf{R}^n$), we can apply Theorem \ref{Theor1} in the intervals $(0,t_1),(t_1,t_2),...,(t_m,T)$. More precisely, under the assumptions of Theorem \ref{Theor1} we can show
	  \[
	  u^\theta \displaystyle\mathrel{\mathop{\xrightarrow{\mathmakebox[1em]{}}}^{}_{\mathrm{\theta\rightarrow+\infty}}} d\quad \text{ locally uniformly in } \mathbf{R}^n\times \left( (0,t_1)\cup (t_1,t_2)\cup ... \cup (t_m,T)\right) .
	  \]
	  Example \ref{exmp:discontdense} in the next subsection shows that we can construct an evolution $\{\Gamma _t\} _{t\in [0,T)}$ for which the associated distance function has discontinuities for each $t\in \mathbf{Q}$. Therefore, the idea described above cannot be applied.
 \end{enumerate}

\end{rem}
The next proposition gives a sufficient condition for the non-extinction condition.
\begin{pro}\label{prop:suf_nonE}
Let $t \in (0,T)$.
If $x \in \overline{\mathrm{int}(\Gamma_t)}$, then $x \not \in E_t$.
\end{pro}

\begin{proof}
Let $x \in \overline{\mathrm{int}(\Gamma_t)}$.
Then there exists a sequence
$\{ x_j \}_{j=1}^{\infty} \subset \mathrm{int}(\Gamma_t)$
that converges to $x$ as $j \to \infty$.
Set $\varepsilon_j:=\mathrm{dist}(x_j, \partial \Gamma_t)$,
which converges to $0$ as $j \to 0$.
Since we have $\overline{B_{\varepsilon_j}(x_j)} \subset \Gamma_t$,
Proposition \ref{prop:FP_ball} (2) implies that
$\overline{\mathcal{C}_{(x_j,t)}^{\varepsilon_j}} \subset \Gamma$.
In particular $x_j \in \Gamma _{t+(\varepsilon_j/L_2)}$,
which is the vertex of the cone,
and consequently we see that $x$ is a non-extinction point.
\end{proof}
\begin{rem}
The converse of the assertion of Proposition \ref{prop:suf_nonE} is not true in general.
In fact, it is easy to construct the interface such that 
$\Gamma_t=\{ 0 \}$ for all $t \in (0,T)$.
Any $x \in \Gamma_t$ is a non-extinction point,
but $\mathrm{int}(\Gamma_t)=\emptyset$.
\end{rem}
\begin{rem}
The opposite notion of an extinction point 
is an emerging point, which is defined as follows:
Let $x \in \Gamma _t$.
We say that $x$ is an {\em emerging point} 
if there exist $\varepsilon,\delta >0$ such that 
$\overline{B_{\varepsilon}(x)} \times [t-\delta,t) \subset D$.
However, the property of finite propagation implies that
there are no emerging points.
Suppose that $x \in \Gamma_t$ is an emerging point, i.e.,
$\overline{B_{\varepsilon}(x)} \times [t-\delta,t) \subset D$
for some $\varepsilon,\delta >0$.
Choose $M \Geqq 1$ large so that 
$x \in \overline{\mathcal{C}_{(x,t-(\delta/M))}^{\varepsilon}}$.
This cone is a subset of $D$ by Proposition \ref{prop:FP_ball} \eqref{itm:FP_1}.
Thus $x \in D$, a contradiction.
\end{rem}
\subsection{An Example}

In this subsection we present an example which shows that the idea presented in Remark \ref{rem:conection_main_theorem} \eqref{itm:rem_connection_3} can not be applied, even if we restrict the evolutions to move inside a bounded domain instead of $\mathbf{R}^n$.
\begin{ex}[A zero level set vanishing for all $t\in \mathbf{Q}\cap (0,T)$]\label{exmp:discontdense}
Let $\mathbf{Q}\cap (0,T)=\lbrace t_1,t_2,...\rbrace$.\\
\textbf{Case 1. In $\mathbf{R}^n$.}\par 
Consider disjoint cubes with sides of length at least $2t_n$ for $n=1,2,...$. Then inside every each one of them, we fit a circle $B_{t_n}$ of radius $t_n$. The evolution of these circles under the equation
\begin{equation}\label{eq23.55*}
V=-1
\end{equation}
where $V$ is the normal velocity (with normal pointing to the exterior of the circles),
is given by
\[
\frac{d}{dt}R(t)=-1.
\]
Here $R(t)$ is the radius of the circles. Notice that the evolution of \eqref{eq23.55*} is the same as the zero level set of the solution $u$ of the problem
 \begin{equation*}
  \left\{
\begin{aligned}
& u_t=-|\nabla u| & \text{ in }& \mathbf{R}^n\times (0,T) , \\
& u(x,0)=u_0(x)  & \text{ in }&\mathbf{R}^n 
\end{aligned}
\right.
    \end{equation*}
if $u_0$ is for example the signed distance function to the circles $B_n$, with positive values in the interior of the circles. For a proof of equivalence of the two evolutions see for example \cite[Section 4.2.3 and 4.2.4]{Giga.See}.
 Then $R(t)=t_n-t$ and the extinction time of the circles is $t=t_n$.\\
\textbf{Case 2. In a bounded domain.}\par 
Let $\Omega$ be a bounded open set. For every $n\in \mathbf{N}$ we  can find points $x_n\in \Omega $ and positive numbers $\varepsilon _n$ such that $B_{\varepsilon _n} (x_n)\subset \Omega$ with $\overline{B_{\varepsilon _n} (x_n)} \cap \overline{B_{\varepsilon _m}(x_m)}=\emptyset \text{ for } n\neq m$.
 Then for $a_n=\varepsilon _n/2(T+1/2)$ we have $R_n:=t_na_n<\varepsilon _n/2$ and $B_{R_n}(x_n)\subset B_{\frac{\varepsilon _n}{2}}(x_n)$. We then define
 \[ c_n (x)=\begin{cases} 
a_n & \mbox{ in } \ B_{\varepsilon _n/2}(x_n), \\
2a_n-\frac{2a_n}{\varepsilon _n}|x-x_n| 
& \mbox{ in } \ B_{\varepsilon _n}(x_n) \setminus B_{\varepsilon _n/2}(x
_n),\\
0 & \mbox{ else }
\end{cases} \]
 and the velocity
\begin{equation*}
  c(x)=\sup_n c_n(x),\quad \text{ for } x\in \mathbf{R}^n.
    \end{equation*}
    As in Case 1 we consider the problem
     \[ 
\begin{cases} 
 u_t=-c(x)|\nabla u|  & \mbox{ in } \ \mathbf{R}^n\times (0,T), \\
u(x,0)=u_0(x)& \mbox{ in } \ \mathbf{R}^n.
\end{cases} 
\]
Here $u_0$ is the signed distance function from the set $\bigcup_{n\in \mathbf{N}}\partial B_{R_n}(x_n)$ with positive values in each $B_{R_n}(x_n)$.
The extinction time of $\partial B_{R_n}(x_n)$ is as in the first case $t=t_n$.\\

\end{ex}

 \section{Homogenization}\label{section:homogen}

We conclude this paper by proving Theorem \ref{Theor2}.
Let us consider 
\begin{equation}\label{eq:resc_init_eq}
u_t=H_1 \left( x,\frac{t}{1+\theta},\nabla u \right)
\end{equation}
and 
\begin{equation}\label{eq:resc_init_eq2}
u_t=H_2 (u,\nabla u),
\end{equation}
where $\theta=k_2/k_1$ is as in \eqref{eq:combined_hamilt}.
By the assumptions on $H_1$ and $H_2$, 
the classical comparison and existence results still hold for the problems \eqref{eq:resc_init_eq}, \eqref{eq:initin} and \eqref{eq:resc_init_eq2}, \eqref{eq:initin}.

To solve the problem \eqref{eq:HJe}, \eqref{eq:initin}
we use the notion of the iterative solution
which was introduced in Remark \ref{Rem:iterative_sol}.
By the comparison and existence results for \eqref{eq:resc_init_eq} and \eqref{eq:resc_init_eq2},
we see that \eqref{eq:HJe}, \eqref{eq:initin} admits a unique continuous iterative solution.

\subsection{Hamiltonians discontinuous in time}\label{subsect:Discontin_Hamilt}

Since the Hamiltonian $H_{12}$ is now discontinuous with respect to time,
we have to be careful about the proof of our homogenization result.
We do not use the notion of viscosity solutions introduced in Definition \ref{defi1},
where the upper- and lower semicontinuous envelopes are used for the equation, because otherwise
we could not estimate 
the difference between $(H_{12})^*$ and $(H_{12})_*$.
Thus we first discuss removability of the upper- and lower star of the equation
as well as a connection between the iterative solution 
and the different notions of viscosity solutions of \eqref{eq:HJe}.

In this section we call $u$ 
a {\em viscosity subsolution} (resp. {\em supersolution})
{\em with star}
if it is a viscosity subsolution (resp. supersolution)
in the sense of Definition \ref{defi1}.
Also, we say that $u$ is 
a {\em viscosity subsolution} (resp. {\em supersolution})
{\em without star}
if it satisfies the viscosity inequality \eqref{eq5.5}
with $F$ instead of $F^*$ (resp. $F_*$).
Note that, since $F_* \leq F \leq F^*$,
a viscosity subsolution (resp. supersolution) without star is always 
a viscosity subsolution (resp. supersolution) with star.
Namely, a notion of viscosity solutions without star is stronger than that with star.

\begin{theo}\label{thm:withoutstar}
Let $u^{\varepsilon}$ be the iterative solution of \eqref{eq:HJe}, \eqref{eq:initin}.
\begin{enumerate}[label=(\arabic*), ref=\arabic*]
	\item\label{itm:homog_1} $u^{\varepsilon}$ is a viscosity solution of \eqref{eq:HJe}, \eqref{eq:initin} without star.
	\item\label{itm:homog_2} If $v$ is a viscosity solution of \eqref{eq:HJe}, \eqref{eq:initin} with star,
then $v=u^\varepsilon$ in $\mathbf{R}^n \times (0,T)$.
\end{enumerate}
\end{theo}

Theorem \ref{thm:withoutstar} \eqref{itm:homog_1} asserts that
$u^{\varepsilon}$ is a viscosity solution in $\mathbf{R}^n \times (0,T)$
not only in the sense with star but also in the sense without star. 
In other words, existence of solutions is established in both the cases.
On the other hand, \eqref{itm:homog_2} is concerned with uniqueness of solutions
since it asserts that any solution should be equal to $u^{\varepsilon}$.
In the sense with star, Perron's method (see Theorem \ref{the2}) gives 
a viscosity solution $u_P$ of \eqref{eq:HJe}, \eqref{eq:initin} which is not necessarily continuous.
By \eqref{itm:homog_2} we see that $u_P=u^{\varepsilon}$, and therefore
$u_P$ is also a viscosity solution of \eqref{eq:HJe}, \eqref{eq:initin} without star
and an iterative solution as well.

\begin{proof}
\eqref{itm:homog_1}
We apply the fact that the viscosity property is extended 
up to the terminal time (\cite[Section 7]{Chen_Giga_Goto}).
Since $u^{\varepsilon}$ is a viscosity solution of \eqref{eq:resc_init_eq}
in $\mathbf{R}^n \times (0,k_1 \Delta t)$,
we see that $u^{\varepsilon}|_{\mathbf{R}^n \times (0,k_1 \Delta t]}$ 
is a viscosity subsolution of \eqref{eq:resc_init_eq}
in $\mathbf{R}^n \times (0,k_1 \Delta t]$.
This implies that $u^{\varepsilon}$ is a viscosity subsolution of \eqref{eq:HJe}
without star on $\mathbf{R}^n \times \{k_1 \Delta t\}$.
Arguing in the same way on $\mathbf{R}^n \times \{ t \}$
with $t=\varepsilon, \varepsilon+k_1\Delta t, 2\varepsilon, \dots$, 
we conclude that $u^{\varepsilon}$ is a viscosity subsolution 
of \eqref{eq:HJe} without star.
The proof for supersolution is similar.

\eqref{itm:homog_2}
Since $v$ and $u^{\varepsilon}$ are, respectively, 
a viscosity subsolution and a supersolution of \eqref{eq:resc_init_eq}
in $\mathbf{R}^n \times (0,k_1 \Delta t)$,
the comparison principle for \eqref{eq:resc_init_eq} implies that
$v^* \leq u^{\varepsilon}$ in $\mathbf{R}^n \times (0,k_1 \Delta t)$.
If we prove $v^* \leq u^{\varepsilon}$ on $\mathbf{R}^n \times \{ k_1 \Delta t\}$,
we then have $v^* \leq u^{\varepsilon}$ in $\mathbf{R}^n \times (k_1 \Delta t,\varepsilon)$
by the comparison principle for \eqref{eq:resc_init_eq2}.
Iterating this argument, we finally obtain 
$v^* \leq u^{\varepsilon}$ in $\mathbf{R}^n \times (0,T)$.
In the same manner, we derive $u^{\varepsilon} \leq v_*$ in $\mathbf{R}^n \times (0,T)$,
and hence $u^{\varepsilon} =v$ in $\mathbf{R}^n \times (0,T)$.

It remains to prove that $v^*(x,k_1 \Delta t) \leq u^{\varepsilon}(x,k_1 \Delta t)$ for $x \in \mathbf{R}^n$.
We now use the fact that $v^*$ is left accessible
(\cite[Section 2, 9]{Chen_Giga_Goto}), i.e.,
there exists a sequence $\{ (x_j,t_j) \}_{j=1}^{\infty}$ such that
$t_j < k_1 \Delta t$ for all $j \geq 1$, 
$(x_j,t_j) \to (x,k_1 \Delta t)$ and 
$v^*(x_j,t_j) \to v^*(x,k_1 \Delta t)$ as $j\to \infty$.
Therefore, taking the limit in $v^*(x_j,t_j) \leq u^{\varepsilon}(x_j,t_j)$
gives $v^*(x,k_1 \Delta t) \leq u^{\varepsilon}(x,k_1 \Delta t)$.
\end{proof}

\begin{rem}\label{rem:comparison_iterartively}
The same argument in the proof of \eqref{itm:homog_2} yields
the comparison principle for \eqref{eq:HJe}.
Namely, if $u$ and $v$ are, respectively, 
a subsolution and a supersolution of \eqref{eq:HJe} with star
such that $u^*(\cdot ,0)\leq v_* (\cdot ,0)$ in $\mathbf{R}^n$, 
then $u^*\leq v_*$ in $\mathbf{R}^n\times[0,T)$.
Also, similar arguments allow us to prove a local version of the comparison principle.
Let $(x,t)\in \mathbf{R}^n \times (0,T)$ and $r>0$.
If $u$ and $v$ are a subsolution and a supersolution of \eqref{eq:HJe}
with star in $B_r(x)\times (t-r,t+r)=:C$, respectively, 
with $u^* \leq v_*$ on $\partial_P C$, 
then $u^* \leq v_*$ in $C$.
Here by $\partial _P$ we denote the parabolic boundary,
that is, for $\Omega \subset \mathbf{R}^n$ and $a<b$, 
\[ \partial _P (\Omega \times (a,b))
:=(\partial \Omega \times [a,b)) \cup (\Omega \times \{ a \} ). \]
\end{rem}

\begin{rem}
See \cite{MR845397, MR2399437}
for more results concerning Hamiltonians discontinuous in time.
\end{rem}

\subsection{Cell problems}
We study an one-dimensional cell problem with discontinuity,
whose solution and eigenvalue will be needed in the proof of our homogenization result.
Consider
\begin{equation}
v'(\tau)+\lambda=H(\tau ) \quad \mbox{in} \ \mathbf{T},
\label{eq:CellPB}
\end{equation}
where $\mathbf{T}=\mathbf{R}/\mathbf{Z}$ is the one-dimensional torus,
$H \in L^1(\mathbf{T})$ and $\lambda \in \mathbf{R}$.
Although we only need to study piecewise continuous $H$
for our homogenization result,
we here take it as a $L^1$-function since the technical aspects of the proof allow us to generalize $H$ without any additional effort.
For the special case where $H$ is piecewise continuous, 
see Remark \ref{rem:piewh}.
We define
\[ H^{\#}(\tau):= \limsup_{k \downarrow 0} 
\left( \frac{1}{k} \int_{\tau-k}^{\tau} H(s)d s \right), \quad 
H_{\#}(\tau):= \liminf_{k \downarrow 0} 
\left( \frac{1}{k} \int_{\tau-k}^{\tau} H(s)d s \right). \]
\begin{lem}[Solvability of the cell problem]\label{lem:Cellpb}
We set 
\begin{equation} 
\lambda := \int_0^1 H(s) ds, \quad
v(\tau):=v(0)-\lambda \tau + \int_0^{\tau} H(s)ds. 
\label{eq:lam_v}
\end{equation}
Then $v$ is a viscosity solution of \eqref{eq:CellPB} in the following sense:
If $\max_{\mathbf{T}}(v-\phi)=(v-\phi)(\tau_0)$
(resp. $\min_{\mathbf{T}}(v-\phi)=(v-\phi)(\tau_0)$)
for $\tau_0 \in \mathbf{T}$ and $\phi \in C^1 (\mathbf{T})$,
then
\begin{equation}
\phi'(\tau_0)+ \lambda \leq H_{\#}(\tau_0) 
\quad \mbox{(resp. $\phi'(\tau_0)+ \lambda \geq H^{\#}(\tau_0)$)}.
\label{eq:visineq_sharp}
\end{equation}
\end{lem}
\begin{proof}
1. 
We first note that $v$ is a periodic function thanks to the choice of $\lambda$.
Indeed, for all $\tau \in \mathbf{R}$ and $m \in \mathbf{Z}$, 
we observe
\begin{align*}
v(\tau +m)
&=v(0)-\lambda (\tau +m) + \int_0^{\tau +m} H(s)ds \\
&=v(0)-\lambda (\tau +m) + \left( \lambda m +\int_0^{\tau} H(s)ds \right) \\
&=v(0)-\lambda \tau + \int_0^{\tau} H(s)ds \\
&=v(\tau).
\end{align*}
Thus $v$ is periodic.

2.
Take $\tau_0 \in \mathbf{T}$ and $\phi \in C^1 (\mathbf{T})$ such that
$\max_{\mathbf{T}}(v-\phi)=(v-\phi)(\tau_0)$.
For $k>0$ we have 
\[ \frac{\phi(\tau_0)-\phi(\tau_0-k)}{k}
\leq \frac{v(\tau_0)-v(\tau_0-k)}{k}
=-\lambda+\frac{1}{k} \int_{\tau_0-k}^{\tau_0} H(s)d s. \]
Taking $\liminf_{k \downarrow 0}$ implies
the first inequality in \eqref{eq:visineq_sharp}.
A similar argument shows that $v$ is a supersolution.
\end{proof}
\begin{rem}\label{rem:piewh}
Let $0=\tau_0<\tau_1< \dots <\tau_N=1$ be a partition of $[0,1]$
and assume that $H \in L^1(\mathbf{T})$ is continuous on each $(\tau_i,\tau_{i+1}]$.
Then we have $H^{\#}=H_{\#}=H$, and consequently 
the viscosity inequalities in \eqref{eq:visineq_sharp} become
\[ \phi'(\tau_0)+ \lambda \leq H(\tau_0) 
\quad \mbox{(resp. $\phi'(\tau_0)+ \lambda \geq H(\tau_0)$)}. \]
In other words, $v$ given by \eqref{eq:lam_v}
is a viscosity solution of \eqref{eq:CellPB} without star.
\end{rem}
\subsection{Proof of homogenization}

\begin{proof}[Proof of Theorem \ref{Theor2}]
1.
Let $u^\varepsilon$ be the iterative solution of \eqref{eq:HJe} and \eqref{eq:initin}.
We denote by $\overline{u}$ and $\underline{u}$
the upper- and lower half-relaxed limit of $u^{\varepsilon}$ respectively, 
i.e., $\overline{u}=\limsup_{\varepsilon \to 0}^* u^\varepsilon$ 
and $\underline{u}=\liminf_{* \varepsilon \to 0}u^\varepsilon$.
Since
the functions $u_0(x)-Kt$ and $u_0(x)+Kt$ with $K>0$ large are, 
respectively, a subsolution and a supersolution of \eqref{eq:HJe},
it follows from comparison that $u_0(x)-Kt \le u^{\varepsilon} (x,t)\leq u_0(x)+Kt$.
This implies $-\infty <\underline{u}\leq \overline{u}<+\infty$
and $\overline{u}(x,0)=\underline{u}(x,0)=u_0(x)$.

2.
Let us show that $\overline{u}$ is a subsolution of \eqref{eq:mean_hmilt_eq}.
Let $\phi$ be a test function for $\overline{u}$ at $(\hat{x},\hat{t})$ from above, i.e,
\begin{align}
&\overline{u} <\phi \quad \text{in} \ 
(B_R(\hat{x}) \times (\hat{t}-R,\hat{t}+R)) \setminus \{ (\hat{x},\hat{t}) \},
\label{eq:test_function1} \\
&\overline{u}(\hat{x},\hat{t})=\phi (\hat{x},\hat{t})
\label{eq:test_function2}
\end{align}
for some $R>0$ such that $0<\hat{t}-R<\hat{t}+R<T$.
We set 
\[ H(\tau):=H_{12}(\hat{x}, \hat{t}, \tau, \hat{\phi}, \nabla \hat{\phi}), \] 
where $\hat{\phi}=\phi (\hat{x},\hat{t})$ and 
$\nabla\hat{\phi}=\nabla\phi (\hat{x},\hat{t})$.
Then $H \in L^1(\mathbf{T})$ and $H$ is continuous 
on $(0,k_1 \Delta t/\varepsilon]$ and $(k_1 \Delta t/\varepsilon,1]$.
Let $v$ and $\lambda$ be as in \eqref{eq:lam_v}.
By Remark \ref{rem:piewh} we see that 
$v$ is a viscosity solution of \eqref{eq:CellPB} without star.
Noting that $\theta =k_2 / k_1$, we observe  
\begin{align*}
\int_0^1 H_{12}(x,t,\tau,r,p) \, d\tau
&= \int_0^{\frac{k_1\Delta t}{\varepsilon}}H_1\left( x,\frac{t}{1+\frac{k_2}{k_1}},p\right)\, d\tau +\int_{\frac{k_1\Delta t}{\varepsilon}}^1 H_2(r,p)\, d\tau \\
&= \frac{k_1}{k_1+k_2}H_1\left( x,\frac{t}{1+\frac{k_2}{k_1}},p\right) +\frac{k_2}{k_1+k_2}H_2(r,p)\\
&= \frac{1}{1+\theta}\left( H_1\left( x,\frac{t}{1+\theta},p\right)+\theta H_2(r,p)\right) \\
&=\bar{H}(x,t,r,p).
\end{align*}
This implies that 
\[ \lambda=\int_0^1 H(\tau) \, d \tau
=\int_0^1 H_{12}(\hat{x}, \hat{t}, \tau, \hat{\phi}, \nabla \hat{\phi}) \, d \tau 
=\bar{H}(\hat{x},\hat{t},\hat{\phi},\nabla \hat{\phi}). \] 
Consequently, $v$ solves
\begin{equation} 
v'(\tau)+\bar{H}(\hat{x},\hat{t},\hat{\phi},\nabla \hat{\phi})
=H_{12}(\hat{x}, \hat{t}, \tau, \hat{\phi}, \nabla \hat{\phi})
\quad \mbox{in} \ \mathbf{T}. 
\label{eq:vHH}
\end{equation}

3.
We want to show that 
\[ \hat{\phi}_t \leq \bar{H}(\hat{x},\hat{t}, \hat{\phi}, \nabla \hat{\phi} ) \]
with $\hat{\phi}_t=\phi_t (\hat{x},\hat{t})$.
Suppose in the contrary that there is $\mu >0$ such that
\begin{equation} 
\hat{\phi}_t \geq 
\bar{H}(\hat{x},\hat{t}, \hat{\phi}, \nabla \hat{\phi} )+\mu. 
\label{eq:phihatmu}
\end{equation}
Let us introduce a perturbed test function.
Define
\[ \phi^\varepsilon (x,t)
=\phi (x,t)+\varepsilon v \left( \frac{t}{\varepsilon} \right). \]
Since $v$ is bounded, 
we see that $\phi^{\varepsilon}$ converges to $\phi$ 
uniformly as $\varepsilon \rightarrow 0$. 
We will show that $\phi ^\varepsilon$ is a supersolution of \eqref{eq:HJe}
in $B_{r}(\hat{x})\times (\hat{t}-r,\hat{t}+r)=:C$, 
where $r \in (0,R)$ is chosen to be small so that
\begin{align}
|\phi_t(x,t)-\hat{\phi}_t| &\le \frac{\mu}{4},
\label{eq:smallr1} \\
\left| H_{12} \left( x,t,\frac{t_0}{\varepsilon},
\phi(x,t),\nabla \phi(x,t) \right)
-H_{12} \left( \hat{x},\hat{t},\frac{t_0}{\varepsilon},
\hat{\phi},\nabla \hat{\phi} \right) \right| &\le \frac{\mu}{8}
\label{eq:smallr2}
\end{align}
for all $(x,t) \in C$.
Although $H_{12}=H_{12}(x,t,\tau,r,p)$ is discontinuous in $\tau$,
\eqref{eq:smallr2} is achieved because we fix $\tau=t_0/\varepsilon$.
More precisely, \eqref{eq:smallr2} is satisfied if 
\begin{align*} 
\left| H_1 \left( x,t,\nabla \phi(x,t) \right)
-H_1 (\hat{x},\hat{t},\nabla \hat{\phi}) \right| 
&\le \frac{\mu}{8}, \\
\left| H_2 \left( \phi(x,t),\nabla \phi(x,t) \right)
-H_2 (\hat{\phi},\nabla \hat{\phi}) \right| 
&\le \frac{\mu}{8} 
\end{align*}
for all $(x,t) \in C$.
Allowing a larger error, 
we are able to replace 
$\phi(x,t)$ on the left-hand side of \eqref{eq:smallr2}
by $\phi^{\varepsilon}(x,t)$ with $\varepsilon >0$ small enough.
Namely, we have 
\begin{equation}
\left| H_{12} \left( x,t,\frac{t_0}{\varepsilon},
\phi^{\varepsilon}(x,t),\nabla \phi(x,t) \right)
-H_{12} \left( \hat{x},\hat{t},\frac{t_0}{\varepsilon},
\hat{\phi},\nabla \hat{\phi} \right) \right| \le \frac{\mu}{4}.
\label{eq:smallr3}
\end{equation}

4.
Let $\psi$ be a test function for $\phi ^\varepsilon$ 
at $(x_0,t_0) \in C$ from below.
Then the function
\[ \tau \mapsto v(\tau)-\frac{1}{\varepsilon}
\left( \psi (x_0,\varepsilon \tau )-\phi (x_0,\varepsilon \tau )\right) \]
has a local minimum at $\tau _0 :=t_0/\varepsilon$. 
Also, from the smoothness of $\phi^{\varepsilon}(\cdot, t_0)$,
it follows that  
\begin{equation}
\nabla \phi^{\varepsilon}(x_0,t_0)
=\nabla \phi(x_0,t_0)=\nabla \psi (x_0,t_0). 
\label{eq:3nabla}
\end{equation}
Since $v$ is a viscosity supersolution of \eqref{eq:vHH}, 
we have
\[ \psi_t (x_0,t_0)-\phi_t (x_0,t_0)
+\bar{H}(\hat{x},\hat{t},\hat{\phi},\nabla \hat{\phi})
\ge H_{12} \left( \hat{x},\hat{t},\frac{t_0}{\varepsilon},\hat{\phi},\nabla \hat{\phi} \right). \]
Let $\phi^{\varepsilon}_0=\phi^{\varepsilon}(x_0,t_0)$,
$\nabla \phi_0=\nabla \phi(x_0,t_0)$ and $\nabla \psi_0=\nabla \psi (x_0,t_0)$.
Applying \eqref{eq:smallr1}, \eqref{eq:smallr3} and \eqref{eq:phihatmu}
to the above inequality,
we compute
\begin{align} 
\psi_t (x_0,t_0) 
&\ge \phi_t (x_0,t_0) 
-\bar{H}(\hat{x},\hat{t},\hat{\phi},\nabla \hat{\phi})
+H_{12} \left( \hat{x},\hat{t},\frac{t_0}{\varepsilon},\hat{\phi},\nabla \hat{\phi} \right) \notag \\
&\ge \hat{\phi}_t -\frac{\mu}{4}
-\bar{H}(\hat{x},\hat{t},\hat{\phi},\nabla \hat{\phi})
+H_{12} \left( x_0,t_0,\frac{t_0}{\varepsilon},\phi^{\varepsilon}_0, \nabla \phi_0 \right)
-\frac{\mu}{4} \notag \\
&\ge H_{12} \left( x_0,t_0,\frac{t_0}{\varepsilon},\phi^{\varepsilon}_0, \nabla \psi_0 \right)
+\frac{\mu}{2}. \label{eq:visc_estim}
\end{align}
For the last inequality we have used \eqref{eq:3nabla}.
The above inequality shows that $\phi^{\varepsilon}$ is a supersolution.
Moreover, since $H_{12}=H_{12}(x,t,\tau,r,p)$ is continuous in the $r$-variable, 
the estimate \eqref{eq:visc_estim} implies that 
there is a small $\eta_0>0$ such that 
$\phi ^\varepsilon -\eta$ is also a supersolution of \eqref{eq:HJe} 
for every $\eta \in (0,\eta_0]$.

5. 
Set 
\[ \delta_0:=-\max_{\partial_P C}(\overline{u}-\phi), \]
which is positive by \eqref{eq:test_function1}.
Also, let $\delta:=\min \{ \delta_0/2, \ \eta_0 \}$.
We then have
\[ \max_{\partial_P C} (u^{\varepsilon}-\phi^{\varepsilon}) \le -\delta, \]
i.e., 
$u^{\varepsilon} \le \phi^{\varepsilon}-\delta$
on $\partial_P C$ for $\varepsilon>0$ small enough.
We now apply the comparison principle 
for a subsolution $u^{\varepsilon}$
and a supersolution $\phi^{\varepsilon}-\delta$ of \eqref{eq:HJe} 
to obtain 
$u^{\varepsilon} \le \phi^{\varepsilon}-\delta$ in $C$.
Taking $\limsup^*_{\varepsilon \to 0}$ at $(\hat{x},\hat{t})$,
we see $\overline{u}(\hat{x},\hat{t}) \le \phi(\hat{x},\hat{t})-\delta$.
This is a contradiction to \eqref{eq:test_function2},
and hence $\overline{u}$ is a subsolution of \eqref{eq:mean_hmilt_eq}.

6.
Similarly we show that $\underline{u}$ is a supersolution of \eqref{eq:mean_hmilt_eq},
and therefore $\overline{u}=\underline{u}$ by comparison.
This implies the locally uniform convergence of $u^{\varepsilon}$ to 
the unique viscosity solution $\bar{u}^{\theta}$ of \eqref{eq:mean_hmilt_eq} and \eqref{eq:initin}.
\end{proof}

\begin{rem}
As long as the comparison principle is true,
this homogenization result still holds for more general equations 
with $H_1$ and $H_2$ which are not necessarily of the forms 
$H_1=H_1(x,t,p)$ and $H_2=H_2(r,p)$.
\end{rem}
\begin{appendix}
\section{Lipschitz continuity of solutions}\label{appendix}

The properties of the solution of the problem \eqref{eq6} and \eqref{eq:initin} 
might come in handy when studying numerical results.
For this reason we prove here a Lipschitz estimate for the solution,
under the assumption that the initial datum is Lipschitz continuous. 
We also give an explicit representation of the Lipschitz constant 
in terms of the Lipschitz constant of the initial datum 
and the Lipschitz constant of the Hamiltonian $H_1$ 
denoted by $D(t)$ as in \eqref{itm6'}. 
Although there are plenty of results in the literature 
concerning the Lipschitz continuity of viscosity solutions, 
a Lipschitz estimate for the Hamiltonians which are being studied in this paper 
does not exist up to the authors' knowledge. 
Moreover we are more concerned in a Lipschitz constant 
that does not depend on the parameter $\theta$.

\begin{proof}[Proof of Proposition \ref{prop2}]
1.
Let $\Phi (x,y,t)=u(x,t)-u(y,t)-L(t)|x-y|$ 
for $x,y\in \mathbf{R}^n$ and $t\in [0,T)$. 
We proceed by contradiction.
Suppose that
\[ M =\sup_{x,y\in \mathbf{R}^n,t\in [0,T)} \Phi (x,y,t)>0. \]

Since $u$ has at most linear growth (Theorem \ref{the2}) 
and $u_0$ is Lipschitz continuous with Lipschitz constant $L_0$, 
we have
\begin{align}
u(x,t)-u(y,t)
&\leq u_0(x)+Kt-(u_0(y)-Kt) \notag \\
&\leq C_T+L_0|x-y|
\label{eq:uuCL}
\end{align}
with $C_T=2KT$.
We define
\[ \Phi_{\sigma} (x,y,t) 
=u(x,t)-u(y,t)-L_{\alpha}(t)|x-y|-\frac{\eta}{T-t}-\alpha (|x|^2+|y|^2), \]
where 
\begin{align*}
L_\alpha (t)&=\max \left\{ L_0, \, 1+2\sqrt{\alpha C_T} \right\} 
e^{\int_0^t (D(s)+\mu (s))\, ds}, \\
\mu (t)&=2\sqrt{\alpha C_T} D(t).
\end{align*}
Set
\[ M_\sigma =\sup_{x,y\in \mathbf{R}^n,\, t\in [0,T)}\Phi _\sigma (x,y,t) \]
for $\sigma =(\eta ,\alpha )$.
Since $u$ has at most linear growth,
there are $x_\sigma , y_\sigma \in \mathbf{R}^n$ and $t_\sigma \in [0,T)$ such that
\[ M_\sigma =\Phi _\sigma (x_\sigma,y_\sigma,t_\sigma). \]

2.
By the definition of $M$, for every $\delta >0$ 
there are $x_\delta ,y_\delta$ and $t_\delta$ such that 
\[ \Phi (x_\delta, y_\delta, t_\delta ) \geq M -\delta. \] 
Since $L_\alpha(t) \rightarrow L(t)=\max \{ L_0, \, 1 \} 
e^{\int_0^t D(s)\, ds}$ uniformly in $t$ as $\alpha \rightarrow 0$, 
there is $\varepsilon >0$ small enough, independent of $\delta$, such that
\[ -L_\alpha (t_\delta )>-\varepsilon -L(t_\delta ). \]
Then for $\delta =M/4$ we have
\begin{align*}
M_\sigma &\geq \Phi_\sigma (x_\delta ,y_\delta ,t_\delta ) \\
&=\Phi (x_\delta ,y_\delta ,t_\delta )-\frac{\eta}{T-t_\delta}
-\alpha (|x_\delta |^2+|y_\delta |^2)-\varepsilon |x_\delta -y_\delta |\\
&\geq M-\delta - \frac{\eta}{T-t_\delta}
-\alpha (|x_\delta |^2+|y_\delta |^2)-\varepsilon |x_\delta -y_\delta | \\
&\geq \frac{M}{2}>0
\end{align*}
for $\eta , \alpha ,\varepsilon$ small enough, since $\delta$ is fixed. From this it follows that  
\begin{equation} 
\Phi_{\sigma}(x_{\sigma},y_{\sigma},t_{\sigma})>0. 
\label{eq:phisigp}
\end{equation}

3.
We claim
\begin{equation} 
\alpha |x_{\sigma} |, \, \alpha |y_{\sigma} | \leq \sqrt{\alpha C_T}. 
\label{eq:thepri}
\end{equation}
By \eqref{eq:uuCL} we observe
\[ u(x,t)-u(y,t)-L_\alpha (t)|x-y|
\leq C_T +(L_0-L_\alpha (t))|x-y|\leq C_T \]
for all $(x,y,t)$, and hence we can write
\[ u(x_{\sigma},t_{\sigma})-u(y_{\sigma},t_{\sigma})
-L_\alpha (t_{\sigma})|x_{\sigma}-y_{\sigma}|-\frac{\eta}{T-t_{\sigma}}
-\alpha |y_{\sigma}|^2 \leq C_T. \]
The left-hand side is equal to 
$\Phi_{\sigma}(x_{\sigma},y_{\sigma},t_{\sigma})+\alpha|x_{\sigma}|^2$.
By \eqref{eq:phisigp} we get 
\[ \alpha |x_{\sigma}|^2 \le C_T. \]
Similarly we have $\alpha |y_{\sigma}|^2 \le C_T$,
and these inequalities show \eqref{eq:thepri}.

4.
We prove $t_\sigma >0$ and $x_\sigma \neq y_\sigma$.
Suppose that $t_\sigma =0$.
Then, since $u_0=u( \cdot ,0)$ is Lipschitz continuous 
with Lipschitz constant $L_0 \leq L_\alpha (0)$, we have
\[ 0< \Phi _\sigma (x_\sigma ,y_\sigma ,0) 
\leq u(x_\sigma ,0)-u(y_\sigma ,0)-L_\alpha (0)|x_\sigma-y_\sigma|
\leq 0, \]
a contradiction.
Since $\Phi _\sigma (x_\sigma ,x_\sigma ,t_\sigma )<0$,
we have that $x_\sigma \neq y_\sigma $.

5.
Since $x_\sigma \neq y_\sigma $,
we have $|x-y|>0$ in a neighbourhood of $(x_\sigma ,y_\sigma)$.
Therefore, we can apply \cite[Lemma 2]{Crandall_Lions} and get,
for $p_\sigma=(x_\sigma -y_\sigma)/|x_\sigma -y_\sigma |$,
\begin{align}
&\frac{\eta}{T^2}+L'_\alpha (t_\sigma )|x_\sigma -y_\sigma | \notag \\
&\leq \left\{ H_1 (x_\sigma ,t_\sigma ,L_\alpha (t_\sigma ) p_\sigma +2\alpha x_\sigma )
-H_1(y_\sigma ,t_\sigma ,L_\alpha (t_\sigma ) p_\sigma -2\alpha y_\sigma ) \right\} \notag \\
& \quad +\theta \{ \beta (u(x_\sigma ,t_\sigma )h(L_\alpha (t_\sigma ) p_\sigma  +2\alpha x_\sigma )
-\beta (u(y_\sigma ,t_\sigma ))h(L_\alpha (t_\sigma ) p_\sigma -2\alpha y_\sigma ) \} \notag \\
&=: I_1+I_2. \label{eq:last_eq}
\end{align}
We can also rewrite $I_2$ as
\begin{align}
I_2 &= 
\theta \left\{  \beta (u(x_\sigma ,t_\sigma ))
-\beta (u(y_\sigma ,t_\sigma )) \right\} 
h(L_\alpha (t_\sigma )p_\sigma +2\alpha x_\sigma ) \notag \\
& \quad + \theta \beta (u(y_\sigma ,t_\sigma )) \cdot 
\{ h(L_\alpha (t_\sigma )p_\sigma +2\alpha x_\sigma )
-h(L_\alpha (t_\sigma )p_\sigma -2\alpha y_\sigma ) \}. 
\label{eq:rewriteI2}
\end{align}

6.
Let us give estimates of $I_1$ and $I_2$.
Using \eqref{itm6'}, the Lipschitz continuity in $x$ of $H_1$, 
together with \eqref{itm5}, we get
\begin{align*}
I_1&\leq D(t_\sigma )|x_\sigma -y_\sigma | \cdot 
|L_\alpha (t_\sigma ) p_\sigma +2\alpha x_\sigma | \\
&\quad + \{ H_1(y_\sigma ,t_\sigma ,L_\alpha (t_\sigma ) p_\sigma +2\alpha x_\sigma )-H_1(y_\sigma ,t_\sigma ,L_\alpha (t_\sigma ) p_\sigma -2\alpha y_\sigma ) \}.
\end{align*}
The first term on the right-hand side can be estimated 
by \eqref{eq:thepri} and $|p_\sigma |=1$,
while we apply \eqref{itm6.1},
the Lipschitz continuity in $p$ of $H_1$, to the second term.
Then
\[ I_1\leq D(t_\sigma )|x_\sigma -y_\sigma |
\left( L_\alpha (t_\sigma ) +2\sqrt{\alpha C_T} \right)
+2\alpha L_2|x_\sigma +y_\sigma |. \]
By the definition of $\mu$ we have
\begin{equation}
I_1 \le |x_\sigma -y_\sigma | \left( D(t_{\sigma})L_{\alpha} (t_{\sigma})+\mu(t_{\sigma}) \right)
+2\alpha L_2|x_\sigma +y_\sigma |.
\label{eq:I_1_estim}
\end{equation}

We next show that the first term on the right-hand side 
of \eqref{eq:rewriteI2} is not positive.
We first note that $u(x_\sigma ,t_\sigma ) \geq u(y_\sigma ,t_\sigma )$
by \eqref{eq:phisigp}.
Therefore, we have 
$\beta (u(x_\sigma ,t_\sigma ))\geq \beta (u(y_\sigma ,t_\sigma ))$
since $\beta $ is increasing. 
As a second remark, by the definition of $L_\alpha$ and \eqref{eq:thepri},
we have
\[ |L_\alpha (t_\sigma )p_\sigma +2\alpha x_\sigma |
\geq L_\alpha (t_\sigma )-2\sqrt{\alpha C_T}\geq 1. \]
Then the assumption \eqref{eq:hprop3} 
and the continuity of $h$ imply that
$h(L_\alpha (t_\sigma )p_\sigma +2\alpha x_\sigma ) \leq 0$.
According to these two remarks we have
\[ I_2 \leq \theta |\beta (u(y_\sigma ,t_\sigma ))| \cdot
|h(L_\alpha (t_\sigma )p_\sigma +2\alpha x_\sigma )
-h(L_\alpha (t_\sigma )p_\sigma -2\alpha y_\sigma )|. \]
Using the boundedness of $\beta$ and the uniform continuity of $h$, \eqref{hprop1},
we can further estimate the above as follows:
\begin{equation}\label{eq:I_2_esimate}
I_2 \leq  \theta M \omega _h (2\alpha |x_\sigma +y_\sigma |),
\end{equation}
where $M$ is an upper bound for $|\beta|$.

7.
Applying \eqref{eq:I_1_estim} and \eqref{eq:I_2_esimate} 
to \eqref{eq:last_eq}, we get
\begin{equation}\label{eq:almost_there}
\frac{\eta}{T^2}+L'_\alpha (t_\sigma )|x_\sigma -y_\sigma |
\leq |x_\sigma -y_\sigma |(L_\alpha (t_\sigma )D(t_\sigma )
+\mu (t_\sigma ))+J,
\end{equation}
where 
\[ J=2\alpha L_2|x_\sigma +y_\sigma |
+\theta M\omega _h(2\alpha |x_\sigma +y_\sigma |). \]
Note that the function $L_\alpha$ has been chosen so that it solves the differential equation 
\[ L_\alpha '(t)=(D(t)+\mu (t))L_\alpha (t). \]
According to this, the estimate \eqref{eq:almost_there} becomes
\[ \frac{\eta}{T^2}
\leq \mu (t_\sigma )(1-L_\alpha (t_\sigma ))|x_\sigma -y_\sigma |+J. \]
Since $L_\alpha >1$, we have
\[ \frac{\eta}{T^2}\leq
J=2\alpha L_2|x_\sigma +y_\sigma |
+\theta M\omega _h(2\alpha |x_\sigma +y_\sigma |). \]
Using \eqref{eq:thepri}, 
we can send $\alpha \rightarrow 0$ and get a contradiction.
\end{proof}

\end{appendix}
\section*{Acknowledgments}
The authors are grateful to Professor R\'egis Monneau
for bringing this problem to their attention and for generously allowing them to develop several of his ideas in this paper.
The authors also thank Professor Cyril Imbert for valuable comments and Dr. Emiliano Cristiani for bringing to their attention the papers \cite{Gomes-Faugeras.1999} and \cite{Delfour-Zolesio.2004}.

Much of the work of the authors was done 
while the first author visited \'Ecole des Ponts ParisTech in 2013 and 2014,
supported by a grant from 
``The research and training center for new development in mathematics''
(Graduate School of Mathematical Science, The University of Tokyo), MEXT, Japan,
the ANR projects HJnet ANR-12-BS01-0008-01
and Grant-in-aid for Scientific Research of JSPS Fellows No.~26-30001.


\begin{thebibliography}{10}

\bibitem{Achdou_Barles_etc}
Y.~Achdou, G.~Barles, H.~Ishii, and G.~L. Litvinov.
\newblock {\em Hamilton-{J}acobi equations: approximations, numerical analysis
  and applications}, volume 2074 of {\em Lecture Notes in Mathematics}.
\newblock Springer, Heidelberg; Fondazione C.I.M.E., Florence, 2013.
\newblock Lecture Notes from the CIME Summer School held in Cetraro, August
  29--September 3, 2011, Edited by Paola Loreti and Nicoletta Anna Tchou,
  Fondazione CIME/CIME Foundation Subseries.

\bibitem{Bardi-Dolcetta}
M.~Bardi and I.~C. Dolcetta.
\newblock {\em Optimal control and viscosity solutions of
  {H}amilton-{J}acobi-{B}ellman equations}.
\newblock Systems \& Control: Foundations \& Applications. Birkh\"auser Boston,
  Inc., Boston, MA, 1997.
\newblock With appendices by Maurizio Falcone and Pierpaolo Soravia.

\bibitem{Barles.1990}
G.~Barles.
\newblock Uniqueness and regularity results for first-order {H}amilton-{J}acobi
  equations.
\newblock {\em Indiana Univ. Math. J.}, 39(2):443--466, 1990.

\bibitem{MR1115933}
G.~Barles and P.~E. Souganidis.
\newblock Convergence of approximation schemes for fully nonlinear second order
  equations.
\newblock {\em Asymptotic Anal.}, 4(3):271--283, 1991.

\bibitem{MR2399437}
M.~Bourgoing.
\newblock Viscosity solutions of fully nonlinear second order parabolic
  equations with {$L^1$} dependence in time and {N}eumann boundary conditions.
\newblock {\em Discrete Contin. Dyn. Syst.}, 21(3):763--800, 2008.


\bibitem{Cannarsa-Sinestrari}
P.~Cannarsa and C.~Sinestrari.
\newblock {\em Semiconcave functions, {H}amilton-{J}acobi equations, and
  optimal control}.
\newblock Progress in Nonlinear Differential Equations and their Applications,
  58. Birkh\"auser Boston, Inc., Boston, MA, 2004.

\bibitem{Chambolle_Novaga}
A.~Chambolle and M.~Novaga.
\newblock Convergence of an algorithm for the anisotropic and crystalline mean
  curvature flow.
\newblock {\em SIAM J. Math. Anal.}, 37(6):1978--1987, 2006.

\bibitem{Chen_Giga_Goto}
Y.~G. Chen, Y.~Giga, and S.~Goto.
\newblock Remarks on viscosity solutions for evolution equations.
\newblock {\em Proc. Japan Acad. Ser. A Math. Sci.}, 67(10):323--328, 1991.

\bibitem{Chopp.1993}
D.~L. Chopp.
\newblock Computing minimal surfaces via level set curvature flow.
\newblock {\em J. Comput. Phys.}, 106(1):77--91, 1993.

\bibitem{Crandall-Ishii-Lions.Ug4}
M.~G. Crandall, H.~Ishii, and P.~L. Lions.
\newblock User's guide to viscosity solutions of second order partial
  differential equations.
\newblock {\em Bull. Amer. Math. Soc. (N.S.)}, 27(1):1--67, 1992.

\bibitem{Crandall_Lions}
M.~G. Crandall and P.~L. Lions.
\newblock On existence and uniqueness of solutions of {H}amilton-{J}acobi
  equations.
\newblock {\em Nonlinear Anal.}, 10(4):353--370, 1986.

\bibitem{Delfour-Zolesio.2004}
M.~C. Delfour and J.~P. Zol{\'e}sio.
\newblock Oriented distance function and its evolution equation for initial
  sets with thin boundary.
\newblock {\em SIAM J. Control Optim.}, 42(6):2286--2304 (electronic), 2004.

\bibitem{Estellers-Zosso-Lai-Osher-Thiran-Bresson.2012}
V.~Estellers, D.~ Zosso, R.~ Lai, S.~ Osher, J.~P.
  Thiran, and X.~Bresson.
\newblock Efficient algorithm for level set method preserving distance
  function.
\newblock {\em IEEE Trans. Image Process.}, 21(12):4722--4734, 2012.

\bibitem{Evans-Soner-Souganidis.1992}
L.~C. Evans, H.~M. Soner, and P.~E. Souganidis.
\newblock Phase transitions and generalized motion by mean curvature.
\newblock {\em Comm. Pure Appl. Math.}, 45(9):1097--1123, 1992.

\bibitem{Evans_Book}
L.~C. Evans.
\newblock {\em Partial differential equations}, volume~19 of {\em Graduate
  Studies in Mathematics}.
\newblock American Mathematical Society, Providence, RI, second edition, 2010.

\bibitem{MR1119185}
Y.~Giga, S.~Goto, H.~Ishii, and M.-H. Sato.
\newblock Comparison principle and convexity preserving properties for singular
  degenerate parabolic equations on unbounded domains.
\newblock {\em Indiana Univ. Math. J.}, 40(2):443--470, 1991.

\bibitem{Giga.See}
Y.~Giga.
\newblock {\em Surface evolution equations}, volume~99 of {\em Monographs in
  Mathematics}.
\newblock Birkh\"auser Verlag, Basel, 2006.
\newblock A level set approach.


\bibitem{Giga-Liu-Mitake.2014}
Y.~Giga, Q.~Liu, and H.~Mitake.
\newblock Singular {N}eumann problems and large-time behavior of solutions of
  noncoercive {H}amilton-{J}acobi equations.
\newblock {\em Trans. Amer. Math. Soc.}, 366(4):1905--1941, 2014.

\bibitem{Gomes-Faugeras.1999}
J.~Gomes and O.~Faugeras.
\newblock Reconciling distance functions and level sets.
\newblock {\em Journal of Visual Communication and Image Representation},
  11(2):209 -- 223, 2000.

\bibitem{Goto_Ishii_Ogawa}
Y.~Goto, K.~Ishii, and T.~Ogawa.
\newblock Method of the distance function to the {B}ence-{M}erriman-{O}sher
  algorithm for motion by mean curvature.
\newblock {\em Commun. Pure Appl. Anal.}, 4(2):311--339, 2005.

\bibitem{Nao-thesis}
N.~Hamamuki.
\newblock {\em A few topics related to maximum principles}.
\newblock PhD thesis, University of Tokyo, 2013.


\bibitem{MR845397}
H.~Ishii.
\newblock Hamilton-{J}acobi equations with discontinuous {H}amiltonians on
  arbitrary open sets.
\newblock {\em Bull. Fac. Sci. Engrg. Chuo Univ.}, 28:33--77, 1985.

\bibitem{Ishii.Perron}
H.~Ishii.
\newblock Perron's method for {H}amilton-{J}acobi equations.
\newblock {\em Duke Math. J.}, 55(2):369--384, 1987.

\bibitem{Ishii.1987}
H.~Ishii.
\newblock A simple, direct proof of uniqueness for solutions of the
  {H}amilton-{J}acobi equations of eikonal type.
\newblock {\em Proc. Amer. Math. Soc.}, 100(2):247--251, 1987.

\bibitem{Monneau-Roquejoffre-RoussierMicho.2013}
R.~Monneau, J.~M. Roquejoffre, and V.~R. Michon.
\newblock Travelling graphs for the forced mean curvature motion in an
  arbitrary space dimension.
\newblock {\em Ann. Sci. \'Ec. Norm. Sup\'er. (4)}, 46(2):217--248 (2013),
  2013.

\bibitem{Osher-Fedkiw.2003}
S.~Osher and R.~Fedkiw.
\newblock {\em Level set methods and dynamic implicit surfaces}, volume 153 of
  {\em Applied Mathematical Sciences}.
\newblock Springer-Verlag, New York, 2003.


\bibitem{Sethian.1999}
J.~A. Sethian.
\newblock {\em Level set methods and fast marching methods}, volume~3 of {\em
  Cambridge Monographs on Applied and Computational Mathematics}.
\newblock Cambridge University Press, Cambridge, second edition, 1999.
\newblock Evolving interfaces in computational geometry, fluid mechanics,
  computer vision, and materials science.

\bibitem{Souganidis}
P.~E. Souganidis.
\newblock Approximation schemes for viscosity solutions of {H}amilton-{J}acobi
  equations.
\newblock {\em J. Differential Equations}, 59(1):1--43, 1985.

\bibitem{Sussman-Fatemi.1999}
M.~Sussman and E.~Fatemi.
\newblock An efficient, interface-preserving level set redistancing algorithm
  and its application to interfacial incompressible fluid flow.
\newblock {\em SIAM J. Sci. Comput.}, 20(4):1165--1191 (electronic), 1999.

\bibitem{Sussman-Smereka-Osher.1994}
M.~Sussman, P.~Smereka, and S.~Osher.
\newblock A level set approach for computing solutions to incompressible
  two-phase flow.
\newblock {\em Journal of Computational Physics}, 114(1):146 -- 159, 1994.







\end{thebibliography}
\end{document}